\documentclass[11pt,a4paper]{article}
\usepackage{bbm}

\usepackage[leqno]{amsmath}
\usepackage{amsfonts}
\usepackage{graphicx}
\usepackage{amsmath}
\usepackage{amssymb}
\usepackage{latexsym}
\usepackage{amsmath, amsfonts,amssymb, amsthm, euscript,makeidx,color,mathrsfs}
\usepackage{enumerate}
\usepackage{tabu}

\usepackage[colorlinks,linkcolor=blue,anchorcolor=green,citecolor=red]{hyperref}

\def\sqr#1#2{{\vcenter{\vbox{\hrule height.#2pt
              \hbox{\vrule width.#2pt height#1pt \kern#1pt \vrule width.#2pt}
              \hrule height.#2pt}}}}
\def\n{\nu}

\def\3n{\negthinspace \negthinspace \negthinspace }
\def\2n{\negthinspace \negthinspace }
\def\1n{\negthinspace }

\def\cl{{\cal l}}
\def\no{\noindent}

\def\bs{\bigskip}
\def\q{\quad}
\def\qq{\qquad}

\def\ep{\epsilon}

\def\cd{\mathbb{C}^d}

\def\bn{\mathcal{N}}

\def\ba{\mathcal{A}}

\def\be{\mathcal{E}}

\def\cd{\cdot}

\def\dim{\hbox{\rm dim$\,$}}

\def\supp{\hbox{\rm supp$\,$}}

\def\cl{\overline}

\def\Re{{\mathop{\rm Re}\,}}
\def\Im{{\mathop{\rm Im}\,}}

\def\({\Big (}
\def\){\Big )}
\def\[{\Big[}
\def\]{\Big]}

\def\={\buildrel \triangle \over =}

\def\be{\begin{equation}}
\def\bel{\begin{equation}\label}
\def\ee{\end{equation}}
\def\bea{\begin{eqnarray}}
\def\eea{\end{eqnarray}}
\def\bt{\begin{theorem}}
\def\et{\end{theorem}}
\def\bc{\begin{corollary}}
\def\ec{\end{corollary}}
\def\bl{\begin{lemma}}
\def\el{\end{lemma}}
\def\bp{\begin{proposition}}
\def\ep{\end{proposition}}
\def\br{\begin{remark}}
\def\er{\end{remark}}
\def\ba{\begin{array}}
\def\ea{\end{array}}
\def\bd{\begin{definition}}
\def\ed{\end{definition}}

\newtheorem{lemma}{Lemma}[section]
\newtheorem{remark}{Remark}[section]
\newtheorem{example}{Example}[section]
\newtheorem{theorem}{Theorem}[section]
\newtheorem{corollary}{Corollary}[section]

\newtheorem{definition}{Definition}[section]
\newtheorem{proposition}{Proposition}[section]

\makeatletter
   
   \@addtoreset{equation}{section}
\makeatother

\begin{document}

\title{Smooth plurisubharmonic exhaustion functions on pseudo-convex domains in infinite dimensions}

\author{Zhouzhe Wang\thanks{
School of Mathematics,
Sichuan University, Chengdu 610064, China.  {\small\it
e-mail:} {\small\tt wangzhouzhe@stu.scu.edu.cn}.}~~~
  and~~~
Xu Zhang\thanks{
School of Mathematics, Sichuan University, Chengdu 610064, China. {\small\it
e-mail:} {\small\tt zhang$\_$xu@scu.edu.cn}.}}

\date{}
\maketitle

\begin{abstract}
In this paper, by modifying significantly the Friedrichs-Gross mollifier
technique and/or using the Lasry-Lions regularization technique together with some carefully chosen cut-off functions, for the first time we construct explicitly smooth exhaustion functions on any open subset and smooth plurisubharmonic exhaustion functions on any pseudo-convex domain in a typical Hilbert space, which enjoy delicate properties needed for the $L^{2}$ method in infinite-dimensional complex analysis.
\end{abstract}

\bs

\no{\bf 2020 Mathematics Subject Classification}. 26E15, 46G05, 46G20.

\bs

\no{\bf Key Words}. Exhaustion function, pseudo-convex domain, plurisubharmonic function,
Friedrichs-Gross mollifier, Lasry-Lions regularization.
\maketitle


\section{Introduction}
In the subject of complex analysis in several  variables, plurisubharmonic functions and pseudo-convex domains play fundamental roles (e.g., \cite{Hor90}).  Naturally, it is expected to extend these concepts/tools
to infinite-dimensional complex analysis. As far as we know, in \cite{Bre}, H. J. Bremermann first considered plurisubharmonic functions  and  pseudo-convex domains in infinite-dimensional spaces. For further works on this topic, we refer to \cite[Chapters 4 and 5]{Her}, \cite[Chapters VIII--X, etc.]{Mujica}, \cite[Chapters 1 and 2]{Noverraz} and so on. In particular, for any nonempty open subset $O$ of a Banach space $X$, combining the related results in \cite{Dineen}, \cite{Gruman}, \cite{Gruman and Kiselman}, \cite{Hirschowitz} and \cite{Noverraz}, people obtained one of the most important results about infinite-dimensional pseudo-convex domains  as follows:
 \begin{theorem}\label{1.1}
 {\rm  (\cite[Theorem 45.8, pp. 324--325]{Mujica})}	Let $X$ be a separable Banach space with the bounded approximation property (defined in \cite[Definition 27.3 (b), p. 196]{Mujica}). Then, $O$ is a pseudo-convex domain of $X$ if and only if one of the following conditions holds true:
 	\begin{itemize}
 		\item $O$ is a domain of existence (defined in \cite[Definition 10.8, p. 82]{Mujica}).
 		\item $O$ is a domain of holomorphy (defined in \cite[Definition 10.4, p. 81]{Mujica}).
 		\item $O$ is holomorphically convex (defined in \cite[Definition 11.3, p. 87]{Mujica}).
 	\end{itemize}
 \end{theorem}

Particularly, the Hilbert space $\ell^{2}$ (See Section \ref{section2} for most of the notations/notions used in this paper) satisfies the condition of Theorem \ref{1.1}. Therefore, many open sets (such as all convex open sets) in $\ell^{2}$ are domains of holomorphy (Thus they are also pseudo-convex by Theorem \ref{1.1}) (See \cite[Corollary 10.7, p. 81]{Mujica}). In the rest of this paper, unless otherwise stated, we will fix an arbitrarily nonempty open set $V$ of $\ell^{2}$.

It is well-known that the classical $L^2$ method (to solve the $\overline{ \partial} $ equation), developed systematically by L.~H\"ormander (\cite{Hor90}), is a basic tool in complex analysis of several variables, which can be applied to solve many problems in complex analysis, complex geometry, algebraic geometry and so on.
One of the key techniques in the $L^{2} $ method is to prove the existence of smooth plurisubharmonic exhaustion functions on pseudo-convex domains (See \cite[Theorem 2.6.11, p. 48]{Hor90}). In the setting of finite dimensional spaces, by means of the translation invariance of the Lebesgue measure and the local compactness, this can be done easily because one may approximate any plurisubharmonic function by smooth plurisubharmonic functions (See \cite[Theorem 2.6.3 at p. 45 and Theorem 1.6.11 at p. 19, and the proof of Theorem 2.6.11 at pp. 48--49]{Hor90}).

Because of the success of the classical $L^2$ estimates in finite dimensions, it is quite natural to expect that such sort of estimates should be extended to the infinite dimensional setting. Nevertheless, due to lack of nontrivial translation invariance measures and local compactness in infinite dimensions, this is a longstanding unsolved problem before our previous works (See \cite{WYZ, YZ} for more details). Particularly, in \cite{WYZ} we establish $L^2$ estimates and existence theorems for the $\overline{\partial}$ operators in each pseudo-convex domain $V$ of $\ell^{2}$ by assuming the existence of the corresponding smooth plurisubharmonic exhaustion function in the space $C^{\infty }_{F}(V)$.

One can find some works on the constructions of smooth plurisubharmonic exhaustion functions on pseudo-convex domains in Banach spaces with unconditional bases (e.g., \cite{Pat, AZ}). Particularly, by \cite[Theorem 1.2]{AZ} and the proof of  \cite[Corollary 34.6, p. 248]{Mujica}, it is easy to see that, for any pseudo-convex domain $V$ of $\ell^{2}$ there is a holomorphic function $f:\;V\to\ell^2$ such that $\left| \left|f\;\right|  \right|$ is a plurisubharmonic exhaustion function on $V$. Nevertheless, it is unclear whether $\left| \left|f\;\right|  \right|\in C^{\infty }_{F}(V)$ or not. Indeed, to the best of our knowledge, completely different from the setting of finite dimensions, in the previous literatures there exist no studies on the existence of smooth plurisubharmonic/exhaustive functions $h(\cdot)$ on general open subsets $U$ of an infinite dimensional space so that $h(\cdot)$ itself and each of its higher order partial derivatives are bounded on any subset which is uniformly included in $U$. The main aim of this paper is to fill this gap and provides particularly a new method to construct explicitly smooth plurisubharmonic exhaustion functions on each pseudo-convex domain of $\ell^{2}$. To do this, the main difficulties are as follows:
\begin{itemize}
		\item As remarked in \cite{YZ}, an essential difficulty in infinite-dimensional analysis is how to treat the differential (especially the higher order differential) of a (possibly vector-valued) function. In this respect, the most popular notions are  Fr\'echet's derivative and G\^ateaux's derivative, which coincide whenever one of them is continuous. However, higher order (including the second order) Fr\'echet's derivatives are multilinear mappings, which are usually not convenient to handle analytically;

	\item Since there exists no nontrivial translation invariant measure in infinite dimensions, one cannot simply regularize a locally integrable function by the usual convolution technique. In \cite[Propositions 2.4 and 3.1]{WYZ}, we have introduced a smoothing function method by projecting first the function under consideration into finite dimensional spaces and then convolving the obtained functions with suitably chosen smooth functions in finite dimensions. However, any function with finite number of variables cannot become an exhaustive function on an infinite-dimensional domain, and therefore such a method does not work for the construction of smooth plurisubharmonic exhaustion functions in $\ell^2$;

	\item A key step in constructing smooth plurisubharmonic exhaustion functions on pseudo-convex domains in the finite dimensional space $\mathbb{C}^m$ ($m\in \mathbb{N}$) is that
\begin{equation}\label{lsh112361xx}
		\inf_{\sum\limits^{m}_{i=1} \left| w_{i}\right|^{2}  =1,w_{1},\cdots ,w_{m}\in \mathbb{C} ,\zeta\in A} \sum_{i,j=1}^m \partial_{i} \overline{\partial_{j} } \eta \left( \zeta\right)  w_{i}\overline{w_{j}} >-\infty ,
\end{equation}
where $A$ is a bounded set in $\mathbb{C}^m$ and $\eta $ is a $C^{\infty }$-function on a domain containing $\cl A$. By the local compactness in finite dimensions, the above inequality (\ref{lsh112361xx}) is trivially correct. However, in the setting of infinite dimensions, what we need is the following:
   \begin{equation}\label{lsh112361}
		\inf_{n\in \mathbb{N} ,\sum\limits^{n}_{i=1} \left| w_{i}\right|^{2}  =1,w_{1},\cdots ,w_{n}\in \mathbb{C} ,\zeta\in A} \sum_{i,j=1}^n  \partial_{i} \overline{\partial_{j} } \eta \left( \zeta\right)  w_{i}\overline{w_{j}} >-\infty,
	\end{equation}
for which $A$ is a bounded set in the infinite dimensional space under consideration and $\eta $ is a $C^{\infty }$-function on a domain containing $\cl A$.
	In infinite dimensional spaces, bounded sets are NOT necessarily relatively compact anymore. On the other hand, the inequality \eqref{lsh112361} also needs to be independent of the dimension $n$. Hence, unlike (\ref{lsh112361xx}), the inequality (\ref{lsh112361}) is NOT guaranteed unless both $A$ and $\eta$ therein are deliberately chosen. This is the most difficult problem to be overcome in the present work.

\end{itemize}
The main contributions of this paper are as follows: \begin{itemize}
	\item Instead of other notions of differentiability in the literatures, in this work we employ the one introduced in our previous papers \cite{WYZ, YZ} (which will be recalled in Section \ref{section2}). By this, together with some modification of the classical Friedrichs mollifier
technique and the Gross mollifier technique introduced in \cite{Gro67}, we successfully construct in Theorem \ref{pmgjsh236118} a smooth approximation sequence for each Lipschitz function in $\ell^2$, which is the basis for us to obtain smooth exhaustive/plurisubharmonic functions (See Theorems \ref{ghqjhshczx23613} and \ref{jntyshxnty23614}, respectively). Combining the main result in this paper (i.e., Corollary \ref{cor5.1}) and \cite[Theorems 4.2 and 5.1]{WYZ}, it is easy to see that our notion of differentiability and the corresponding function spaces are useful to solve some difficult problems arising in infinite-dimensional complex analysis.

	\item  We employ the Lasry-Lions regularization technique, i.e., the method of inf-sup-convolution introduced in \cite{LaL86} (See also \cite [Appendices C, pp. 347--351]{Prato}) to prove that there exists a Lipischitz continuous, semi-anti-plurisubharmonic exhaustion function on any open subset of $\ell^2$ (See Definition \ref{bthshdy23611} and Theorem \ref{shtqj23612}), via which we successfully overcome the problem of requiring a dimension-free, finite lower bound in the second-order differential inequality (\ref{lsh112361}).

	\item As a byproduct, we construct some smooth exhaustion functions (with fine properties) on any open subset of $\ell^{2}$ (See Theorems \ref{ghqjhshczx23613} and \ref{shtqj23612}). On the other hand, in \cite[Definition 2.3]{WYZ} we require the plurisubharmonic exhaustion function $\eta \in C^{\infty }_{F}(V)$ while in Theorem \ref{jntyshxnty23614} we construct such a function $\eta \in C^{\infty }_{F^{\infty}}(V)$. Both the smooth functions introduced by us and the method to construct them have independent interest, which may be applied to solve many other problems, for example, Theorem \ref{jntyshxnty23614} will play a crucial role in our forthcoming work on the regularity of solutions to the $ \overline{ \partial} $ equation in infinite dimensions.
\end{itemize}

The rest of this paper is organized as follows: In Section \ref{section2}, we collect some notations, definitions and some simple results which will be useful later. Section \ref{sec3} is addressed to a study of exhaustion functions and Lipschitz functions, while Section \ref{section4} is mainly for the construction of semi-anti-plurisubharmonic exhaustion functions on any open subset of $\ell^2$. In the last section, i.e., Section \ref{section5}, we shall characterize pseudo-convex domains by smooth plurisubharmonic functions.

\section{Preliminaries}\label{section2}

For the readers' convenience, we quickly recall some notations and definitions in \cite{Mujica, WYZ, YZ}, as well as some other preliminary results.

Denote by $\mathbb{R}$ and $\mathbb{C}$ respectively the sets (or fields, or Euclidean spaces with the usual norms) of all real and complex numbers, and by $\mathbb{N}$ the set of all positive integers. Suppose that $G$ is a
nonempty set and $(G,\cal J)$ is a topological space on $G$. The smallest $\sigma$-algebra $\mathscr{B}(G)$
generated by all sets in $\cal J$ is called the
Borel $\sigma$-algebra of
$G$. For any $\sigma$-finite Borel measure $\nu$ on $G$ and $p\in [1,\infty)$, denote by $L^p(G,\nu)$ the Banach space of all
(equivalence classes of) complex-valued Borel-measurable functions $f$ on $G$ for which $\int_G|f|^pd\nu<\infty$ (Particularly, $L^2(G,\nu)$ is a Hilbert space with the canonical inner product).
For any subset $A$ of $G$, denote by $A^o$ the set of all interior points of $A$, by $\overline{A}$ the closure of $A$ in $G$, by $\partial A$ the boundary of $A$, and  by $\chi_A(\cdot)$ the characteristic function of $A$ in $G$.

The following elementary result must be known but we do not find an exact reference.

\begin{proposition}\label{collected}
Any linear topological space is path connected, and hence connected.
\end{proposition}

\begin{proof}
Let $Y$ be a linear topological space. Pick any $p,q\in Y$. Then, the continuous map $\alpha: [0,1]\to Y$ defined by $\alpha(t)=tp+(1-t)q$ satisfies that $\alpha(0)=p$ and  $\alpha(1)=q$, which indicates that $Y$ is path connected, and hence connected.
\end{proof}

Recall that $O$ is a nonempty open subset of the Banach space $X$. Denote by $C(O)$ the set of all complex-valued continuous functions on $O$. By \cite[Definition 34.3, p. 247]{Mujica} and \cite[Definition 17.1, p. 335]{Rud87} (and noting \cite[Corollary 1.6.5, p. 18]{Hor90}), we have the following concept.

\begin{definition}\label{dchcthhshdy23615}
A real-valued function $f\in C(O)$ is said to be plurisubharmonic (resp., anti-plurisubharmonic) (on $O$) if
		$$
f\left( a\right)  \leqslant \frac{1}{2\pi } \int^{2\pi }_{0} f\left( a+e^{\sqrt{-1}\theta }b\right)  d\theta
 $$
holds for any $a\in O$ and $b\in X$ with $\left\{ a+e^{\sqrt{-1}\theta} b:\;\theta \in  [0,2\pi]\right\}  \subset O$ (resp., $-f$ is plurisubharmonic). Particularly, $f$ is called subharmonic if $X=\mathbb{C}$ and $f$ is plurisubharmonic.
\end{definition}


\begin{proposition}\label{dchcthhshzyxwdxzh23616}
Plurisubharmonic functions enjoy the following properties:
	\begin{itemize}
		\item[$(1)$]	A real-valued function $f\in C(O)$ is plurisubharmonic if and only if so is the restriction of $f$ to $O\bigcap M$ for each finite dimensional subspace $M$ of $X$ (cf., Assertion (c) in \cite[Proposition 34.5, pp. 247--248]{Mujica}).
			\item[$(2)$]	If both $f$ and $g$ are plurisubharmonic functions on $O$ and $a,b\geqslant0$, then so is $af+bg$ (cf., Assertion (a) in \cite[Proposition 34.5, p. 247]{Mujica}).

				\item[$(3)$]	If $f$ is plurisubharmonic on $O$ and $\varphi:\mathbb{R}\rightarrow \mathbb{R}$ is a convex, increasing function, then so is $\varphi(f)$ (cf., \cite[Proposition 34.12, pp. 253--254]{Mujica}).
			
	\end{itemize}
\end{proposition}

The following result is a special case of \cite[Theorem 34.8, pp. 249--250]{Mujica}:
\begin{proposition}\label{dchcthhshjbd23617}
	A real-valued function $f\in C(O)$ is plurisubharmonic if and only if for each $a\in O$, $b\in X$ and $\delta>0$, there exists $r\in (0,\delta)$ such that  $\big\{ a+re^{\sqrt{-1}\theta} b:\;\theta \in  [0,2\pi]\big\}  \subset O$ and
 $$
 f\left( a\right)  \leqslant \frac{1}{2\pi } \int^{2\pi }_{0} f\left( a+re^{\sqrt{-1}\theta }b\right)  d\theta .
 $$
\end{proposition}

 For any $z=x+\sqrt{-1}y\in \mathbb{C}$ with $x,y\in \mathbb{R}$, denote by $\overline{z}$ the complex conjugation of $z$, and write $\Re z\triangleq\frac{z+\overline{z} }{2} $, $\Im z\triangleq\frac{z-\overline{z} }{2\sqrt{-1}}$, $\frac{\partial }{\partial z} \triangleq\frac{1}{2}\left( \frac{\partial }{\partial x}-\sqrt{-1}\frac{\partial }{\partial y}\right)$ and
$\frac{\partial }{\partial \overline{z}} \triangleq\frac{1}{2}\left( \frac{\partial }{\partial x}+\sqrt{-1}\frac{\partial }{\partial y}\right)$.
 By \cite[(3.12), p. 39]{Ohsawa02}, we have
\begin{proposition}\label{cthgh236129}
  Assume that $W$ is a nonempty open set of $\mathbb{C}$ and $f$ is a real-valued, second order continuous differentiable function on $W$.  Then, $f$ is  subharmonic on $W$ if and only if  $$\Delta f\left( z\right)  \triangleq 4\frac{\partial^{2} f\left( z\right)  }{\partial z\partial \overline{z} } \geqslant 0,\q\forall\;z\in W.$$
\end{proposition}

When $\dim X<\infty$, $O$ is called a pseudo-convex domain in $X$ if there exists a plurisubharmonic function $\eta $ on $O$ so that $\{x\in O:\;\eta (x) < t\} $ is relatively compact in $O$ for all $t\in \mathbb{R}$ (See \cite[Theorem 2.6.7, p. 46]{Hor90}). Generally, in view of \cite[Definition 37.3, p. 274]{Mujica} and \cite[Corollary 37.6, p. 277]{Mujica}, we have the following notion.

\begin{definition}\label{xntyddy23618}
	$O$ is called a pseudo-convex domain in $X$, if for each finite dimensional subspace $M$ of $X$, $O\cap M$ is a pseudo-convex domain in $M$.
\end{definition}

 Write $\mathbb{N}_0\triangleq \mathbb{N}\cup \{0\}$ and $\mathbb{C}^{\infty}=\prod\limits_{i=1}^{\infty}\mathbb{C}$, which is the countable Cartesian product of $\mathbb{C}$, i.e., the space of sequences of complex numbers (Nevertheless, in the sequel we shall write an element $\textbf{z}\in \mathbb{C}^{\infty }$ explicitly as $\textbf{z}=( z_{i})^{\infty }_{i=1}$ rather than $\textbf{z}=\{ z_{i}\}^{\infty }_{i=1}$, where $z_i\in\mathbb{C}$ for each $i\in \mathbb{N}$). Clearly, $\mathbb{C}^{\infty}$ can be identified with $(\mathbb{R}^2)^{\infty}=\prod\limits_{i=1}^{\infty}\mathbb{R}^2$. Set
$$
\ell^{2}\triangleq\left\{( z_{i})^{\infty }_{i=1} \in \mathbb{C}^{\infty } :\; \sum\limits^{\infty }_{i=1} \left| z_{i}\right|^{2}  <\infty \right\}.
$$
Then, $\ell^2$ is a separable Hilbert space with the canonic inner product $\left( \cdot,\cdot\right)$ (and hence  the canonic norm $\left| \left|\;\cdot\;\right|  \right|$). \
For any $\textbf{a}\in \ell^{2}$ and $r\in(0,+\infty)$, write
$$B_{r}(\textbf{a})\triangleq \left\{ \textbf{z}\in \ell^{2} :\;\left| \left| \textbf{z}-\textbf{a}\right|  \right|<r\right\}.
$$
We simply write $B_{r}$ for $B_{r}(0)$.

In what follows, we fix a sequence $\{a_i\}_{i=1}^{\infty}$ of positive numbers with
$\sum\limits_{i=1}^{\infty} a_i <1$.
Let
$$
	\mathcal {N}=\prod_{i=1}^{\infty}\frac{1}{2\pi a_{i}^2}e^{-\frac{x_{i}^2+y_{i}^2 }{2a_i^2}}\,\mathrm{d}x_{i}\mathrm{d}y_{i}
$$
be the product measure on $(\mathbb{R}^2)^{\infty}=\mathbb{C}^{\infty}$  (here and henceforth $x_i,y_i\in \mathbb{R}$ for each $i\in \mathbb{N}$). Noting that $\ell^2$ is a Borel subset of $\mathbb{C}^{\infty}$, $\mathcal {N}(\ell^2 )=1$ and	
$\mathscr{B}(\ell^2)=\{F\cap\ell^2:\; F\in \mathscr{B}(\mathbb{C}^{\infty})\}$, we obtain a Borel probability measure $P$ on $\ell^2$ by setting
\begin{equation}\label{230604e0}
P(E)\triangleq \bn(F),\quad\forall\;E\in \mathscr{B}(\ell^2),
\end{equation}
where $F$ is any subset in $\mathscr{B}(\mathbb{C}^{\infty})$ so that $E=F\cap\ell^2$. More generally, for each $n\in\mathbb{N}$, we define a probability measure $\bn^n$ in $(\mathbb{C}^n,\mathscr{B}(\mathbb{C}^n))$ by setting
$$
\bn^n\triangleq \prod_{i=1}^{n}\frac{1}{2\pi a_{i}^2}e^{-\frac{x_{i}^2+y_{i}^2 }{2a_i^2}}\,\mathrm{d}x_{i}\mathrm{d}y_{i}
$$
and a probability measure ${\widehat{ \bn}^n}$ in $(\mathbb{C}^{\infty},\mathscr{B}(\mathbb{C}^\infty))$ by setting
$$
\widehat{\bn}^n\triangleq \prod_{i=n+1}^{\infty}\frac{1}{2\pi a_{i}^2}e^{-\frac{x_{i}^2+y_{i}^2 }{2a_i^2}}\,\mathrm{d}x_{i}\mathrm{d}y_{i}.
$$
Then, we obtain a Borel probability measure $P_n$ on $\ell^2$ by setting
\begin{equation}\label{230604e1}
P_n(E)\triangleq \widehat{\bn}^n(F),\quad\forall\;E\in \mathscr{B}(\ell^2),
\end{equation}
where $F$ is any subset in $\mathscr{B}(\mathbb{C}^{\infty})$ for which $E=F\cap\ell^2$.
Obviously,
$$P=\bn^n\times P_n,\quad\forall\; n\in \mathbb{N}.$$

For $n\in \mathbb{N}$ and $\textbf{z}=(z_{i} )_{i=1}^\infty\in \ell^2$, set
\begin{equation}\label{231128e1}
\textbf{z}_n=(z_{i})_{i=1}^{n},\q \textbf{z}^n=(z_{i} )_{i=n+1}^\infty.
 \end{equation}
By \cite[Proposition 2.4]{WYZ},  we have
\begin{proposition}\label{jwcz236112}
	Suppose that $f\in L^{2}\left( \ell^{2} ,P\right)  $ .  For each $n\in \mathbb{N}$, set
	{\rm	$$
f_n(\cdot)\triangleq \int f(\cdot,\textbf{z}^n)\,\mathrm{d}P_n(\textbf{z}^n).
 $$}
Then
$\lim\limits_{n\rightarrow \infty } \left| \left| f_{n}-f\right|  \right|_{L^{2}\left( \ell^{2} ,P\right)  }  =0$.
\end{proposition}

The following result should be known but we do not find an exact reference.
	\begin{lemma}\label{jbysdsh236115}
	If $\left\{ \theta_{i} \right\}^{\infty }_{i=1}  \subset \mathbb{R}$ and $f\in L^{1}\left( \ell^{2} ,P\right)  $, then $f\left( e^{\sqrt{-1}\theta_{1} }z_{1},e^{\sqrt{-1}\theta_{2} }z_{2},\cdots \right)  \in L^{1}\left( \ell^{2} ,P\right)  $, and
 $$
 \int_{\ell^{2} } f\left( z_{1},z_{2},\cdots \right)  dP\left( z_{1},z_{2},\cdots \right)=\int_{\ell^{2} } f\left( e^{\sqrt{-1}\theta_{1} }z_{1},e^{\sqrt{-1}\theta_{2} }z_{2},\cdots \right)  dP\left( z_{1},z_{2},\cdots \right).
 $$
	\end{lemma}
\begin{proof}
	Let
 $$
 M_{0}\triangleq \left\{ \ell^{2} \cap \left( C_{n}\times \mathbb{C} \times \mathbb{C} \times \cdots \right)  :\;C_{n}\in \mathscr{B}\left( \mathbb{C}^{n} \right),n\in \mathbb{N}  \right\}  .
 $$
 For each $n\in \mathbb{N}$, put (Recall (\ref{231128e1}) for $\textbf{z}_n$)
 $$
 \varphi_{n} \left( \textbf{z}_n\right)  \triangleq \frac{1}{\left( 2\pi \right)^{n}  a^{2}_{1}a^{2}_{2}\cdots a^{2}_{n}} e^{-\sum\limits^{n}_{j=1} \frac{\left| z_{j}\right|^{2}  }{2a^{2}_{j}} }.
 $$
 It is clear that $M_{0}$ is a $\pi$-class, and  by \cite[Exercise 1.10, p. 12]{Prato2006}, $\sigma \left( M_{0}\right)  =\mathscr{B}( \mathbb{C}^{\infty } )  \bigcap \ell^{2} =\mathscr{B}( \ell^{2})   $, where $\sigma \left( M_{0}\right) $ denotes the $\sigma$-field generated by $M_0$.
	
For any given $A\in M_{0}$ and $f(\cdot)=\chi_{A} \left( \cdot\right)$, there exists $n_0\in \mathbb{N}$ and $C_{n_0}\in \mathscr{B}\left( \mathbb{C}^{n_0} \right)$ such that $A=\ell^{2} \cap \left( C_{n_0}\times \mathbb{C} \times \mathbb{C} \times \cdots \right)$, and hence
 $$
		\begin{aligned}
			\int_{\ell^{2} } \chi_{A} \left( z_{1},z_{2},\cdots \right)  dP&=\int_{\mathbb{C}^{n_0} } \chi_{C_{n_0}} \left( \textbf{z}_{n_0}\right) \varphi_{n_0} \left( \textbf{z}_{n_0}\right)  dm_{2n_0}\left(\textbf{z}_{n_0}\right)  \\
       &=\int_{\mathbb{C}^{n_0} } \chi_{C_{n_0}} \left( e^{\sqrt{-1}\theta_{1} }z_{1},e^{\sqrt{-1}\theta_{2} }z_{2},\cdots ,e^{\sqrt{-1}\theta_{n_0} }z_{n_0}\right)  \varphi_{n_0} \left( \textbf{z}_{n_0}\right)  dm_{2n_0}\left( \textbf{z}_{n_0}\right)  \\
       &=\int_{\ell^{2} } \chi_{A} \left( e^{\sqrt{-1}\theta_{1} }z_{1},e^{\sqrt{-1}\theta_{2} }z_{2},\cdots \right)  dP,
		\end{aligned}
	$$
	where the notation $dm_{2n_0}$ stands for the  $2n_0$-dimensional Lebesgue measure.

	Let
 $$
 M\triangleq \left\{ A\in \mathscr{B}( \ell^{2})  :\;P\left( A\right)  =\int_{\ell^{2} } \chi_{A} \left( e^{\sqrt{-1}\theta_{1} }z_{1},e^{\sqrt{-1}\theta_{2} }z_{2},\cdots \right)  dP\right\}  .
 $$
 Clearly, $\ell^{2},\emptyset  \in M$.
	If $A_{n}\in M$ satisfying $A_{n}\subset A_{n+1},n=1,2,\cdots $, then we write $A=\bigcup\limits^{\infty }_{n=1} A_{n}$. By Levi's lemma, we have
 $$
 \begin{aligned}
 P\left( A\right)  =\lim_{n\rightarrow \infty } P\left( A_{n}\right) & =\lim_{n\rightarrow \infty } \int_{\ell^{2} } \chi_{A_{n}} \left( e^{\sqrt{-1}\theta_{1} }z_{1},e^{\sqrt{-1}\theta_{2} }z_{2},\cdots \right)  dP\\
&=\int_{\ell^{2} } \chi_{A} \left( e^{\sqrt{-1}\theta_{1} }z_{1},e^{\sqrt{-1}\theta_{2} }z_{2},\cdots \right)  dP,
 	\end{aligned}
 $$
 which means that  $A\in M$.	
	For any $A,B\in M$ satisfying $B\subset A $, one has
 $$
		\begin{aligned}
			P\left( A\setminus B\right) & =P\left( A\right)  -P\left( B\right)\\
&=\int_{\ell^{2} } \left[\chi_{A} \left( e^{\sqrt{-1}\theta_{1} }z_{1},e^{\sqrt{-1}\theta_{2} }z_{2},\cdots \right)  -\chi_{B} \left( e^{\sqrt{-1}\theta_{1} }z_{1},e^{\sqrt{-1}\theta_{2} }z_{2},\cdots \right)  \right]dP\\
&=\int_{\ell^{2} } \left[\chi_{A\setminus B} \left( e^{\sqrt{-1}\theta_{1} }z_{1},e^{\sqrt{-1}\theta_{2} }z_{2},\cdots \right) \right]dP,
		\end{aligned}
	$$
	hence $A\setminus B\in M$, and $M$ is a $\lambda$-class. By $M_{0}\subset M$ and the Monotone Class Theorem, it follows that $M\supset \sigma \left( M_{0}\right)  =\mathscr{B}( \ell^{2})   $, and hence $M=\mathscr{B}( \ell^{2})  $.

For any $f\in L^1(\ell^2,P)$, without loss of generality we assume that $f\geqslant0$. By \cite[Theorem 1.17, p. 15]{Rud87}, there exists a sequence $\{f_{n}\}_{n=1}^\infty$ of simple measurable functions so that $0\leqslant f_{n}\leqslant f_{n+1}$ for $n\in \mathbb{N}$, and
 $$
 \lim_{n\rightarrow \infty } f_{n}\left( \textbf{z}\right)  =f\left( \textbf{z}\right)  ,\quad\forall\; \textbf{z}\in \ell^{2} .
 $$
By Lebesgue's Monotone Convergence Theorem (e.g., \cite[Theorem 1.26, p. 21]{Rud87}), we see that
 $$
 \begin{aligned}
	\int_{\ell^{2} } f\left( z_{1},z_{2},\cdots \right)  dP&=\lim_{n\rightarrow \infty } \int_{\ell^{2} } f_{n}\left( z_{1},z_{2},\cdots \right)  dP
=\lim_{n\rightarrow \infty } \int_{\ell^{2} } f_{n}\left( e^{\sqrt{-1}\theta_{1} }z_{1},e^{\sqrt{-1}\theta_{2} }z_{2},\cdots \right)  dP\\
&=\int_{\ell^{2} } f\left( e^{\sqrt{-1}\theta_{1} }z_{1},e^{\sqrt{-1}\theta_{2} }z_{2},\cdots \right)  dP,
\end{aligned}
$$
which completes the proof of Lemma \ref{jbysdsh236115}.
\end{proof}

For a nonempty bounded subset $S$ of $V$, write
$$
d\left( S,\partial V\right)  \triangleq \begin{cases}\inf\limits_{\textbf{z}_1\in S,\textbf{z}_2 \in \partial V} \left| \left| \textbf{z}_1-\textbf{z}_2\right|  \right|  ,&V\neq \ell^{2}, \\ +\infty ,&V=\ell^{2} . \end{cases}
$$
Particularly, for $\textbf{z}\in V$, let
\begin{equation}\label{zx230614e1}
d\left( \textbf{z},\partial V\right)\triangleq d\left( \{\textbf{z}\},\partial V\right).
\end{equation}
In view of \cite[Definition 2.1 and Lemma 2.3]{WYZ},  $S$ is said to be uniformly included in $V$, denoted by $S\stackrel{\circ}{\subset} V$, if $d\left( S,\partial V\right)>0$.

Let us recall a notion of differentiability and some spaces of functions introduced in \cite{WYZ, YZ}. Suppose that $f$ is a complex-valued function defined on $V$. As in \cite[p. 528]{YZ}, for each $\textbf{z}=(z_j)_{j=1}^{\infty}=(x_j+\sqrt{-1}y_j)_{j=1}^{\infty}\in V$, we define the partial derivatives of $f(\cdot)$ (at $\textbf{z}$ with respect to $x_j$ and $y_j$) as follows:
$$
\begin{array}{ll}
\displaystyle\frac{\partial f(\textbf{z})}{\partial x_j}\triangleq\lim_{\mathbb{R}\ni \tau\to 0}\frac{f(z_1,\cdots,z_{j-1},z_j+\tau,z_{j+1},\cdots)-f(\textbf{z})}{\tau},\\[3mm]
\displaystyle\frac{\partial f(\textbf{z})}{\partial y_j}\triangleq\lim_{\mathbb{R}\ni \tau\to 0}\frac{f(z_1,\cdots,z_{j-1},z_j+\sqrt{-1}\tau,z_{j+1},\cdots)-f(\textbf{z})}{\tau}.
\end{array}
$$
For each $k\in\mathbb{N}_0$, we denote by $C^{k} ( V)$ the set of all complex-valued functions $f$ on $V$ for which $f$ itself and all of its  partial derivatives up to order $k$ are continuous on $V$ (As usual we simply write $C^0(V)$ as $C(V)$). Write
$$
C^{\infty} ( V)\triangleq \bigcap_{j=1}^{\infty}C^{j} ( V).
$$
For $f\in C(V)$, we call $\supp f\triangleq\overline{\{\textbf{z}\in V:\;f(\textbf{z})\neq 0\}}$ the support of $f$.
For $j\in\mathbb{N}$, write
$$
\partial_{j} \triangleq\frac{1}{2}\left( \frac{\partial }{\partial x_{j}}-\sqrt{-1}\frac{\partial }{\partial y_{j}}\right),\quad
\overline{\partial }_{j} \triangleq\frac{1}{2}\left( \frac{\partial }{\partial x_{j}}+\sqrt{-1}\frac{\partial }{\partial y_{j}}\right).
$$
Denote by $\mathbb{N}_0^{(\mathbb{N})}$ the set of all finitely supported sequences of nonnegative integers, i.e., for each $\alpha =\{\alpha_{j} \}_{j=1}^\infty\in \mathbb{N}^{\left( \mathbb{N} \right)  }_{0} $, one has $\alpha_{m}\in\mathbb{N}_{0}$ for any $ m\in\mathbb{N}$, and there exists $n\in\mathbb{N}$ such that $\alpha_j=0$ for all $j\geqslant n$. Set $\left| \alpha \right|  \triangleq \sum\limits^{\infty }_{j=1} \alpha_{j}$ and write
$$
\begin{array}{ll}
\displaystyle
\partial_{\alpha } \triangleq \partial^{\alpha_{1} }_{1} \partial^{\alpha_{2} }_{2} \cdots \partial^{\alpha_{n} }_{n}  ,\q  \overline{\partial}_{\alpha }  \triangleq \overline{\partial}_{1} ^{\alpha_{1}}  \overline{\partial}_{2} ^{\alpha _{2} } \cdots \overline{\partial}_{n} ^{\alpha _{n} }.
\end{array}
$$
As in \cite{WYZ}, we put
$$
C^{1}_{F}\left( V\right)\triangleq \left\{f\in C^{1}\left( V\right):\;\sup_{S} \left\{ \left| f\right|  +\sum^{\infty }_{j=1} \bigg(\left| \partial_{j}  f\right|^{2}  +\left|\overline{\partial }_{j}  f\right|^{2}\bigg)  \right\} <\infty,\; \forall\; S\stackrel{\circ}{\subset}V\right\}.
$$
For $k\in\mathbb{N}\cup \{\infty\}$, we write $C^{k}_{F}\left( V\right)\triangleq C^{k}\left( V\right)\cap C^{1}_{F}\left( V\right)$ and
$$
\begin{array}{ll}
\displaystyle C^{k}_{F^{k}}\left( V\right)\triangleq \left\{f\in C^{k}(V):\;\partial_{a} \overline{\partial_{\beta } } f\in C^{1}_{F}\left( V\right),\q\forall\;\alpha, \beta\in \mathbb{N}^{\left( \mathbb{N} \right)  }_{0}\text{ with }|\alpha|+|\beta|<k\right\},\\[3mm]
\displaystyle C^{k}_{0}\left( V\right)\triangleq \Big\{f\in C^{k}\left( V\right):\;\supp  f\stackrel{\circ}{\subset}  V\hbox{ and }\sup_{V} \Big| \partial_{\alpha } \overline{\partial}_{\beta } f\Big|  <\infty,\q\forall\;\alpha, \beta\in \mathbb{N}^{\left( \mathbb{N} \right)  }_{0}\text{ with }|\alpha|+|\beta|<k\Big\},\\[3mm]
\displaystyle C^{k}_{0,F}\left( V\right)\triangleq \Big\{f\in C^{k}_{F}\left( V\right):\;\supp  f\stackrel{\circ}{\subset}  V\Big\},\quad C^{k}_{0,F^{k}}\left( V\right)\triangleq C^{k}_{0}\left( V\right)\cap C^{k}_{F^{k}}\left( V\right).
\end{array}
$$

To end this section, motivated by its counterpart in finite dimensions (e.g., \cite[Theorem 2.6.2, p. 44]{Hor90}), we have the following simple result.
\begin{lemma}\label{ghqd236130}
	For any real-valued function $\varphi \in C^{2}\left( V \right)  $, the following conditions are equivalent:
\begin{itemize}
		\item[$(a)$] The function $\varphi$ is plurisubharmonic on $V$;
		\item[$(b)$] For  each $ \textbf{z}\in V$ and $\zeta\in\ell^{2}$,
the function $\varphi \left(  \textbf{z}+\cdot \zeta\right)
$
is subharmonic on $\left\{ \lambda \in \mathbb{C} : \textbf{z}+\lambda \zeta\in V\right\}  $;
		\item[$(c)$] For  each $\textbf{z}\in V$, $n\in \mathbb{N}$ and $\zeta\in\mathbb{C}^{n}$, 	
the function $ \varphi \left(  \textbf{z}+\cd\zeta\right)  $ is subharmonic on $\big\{ \lambda \in \mathbb{C} : \textbf{z}+\lambda \zeta\in$ $ V\big\}  $;
		\item[$(d)$]  $\sum^{n}_{i,j=1} \overline{\partial_{j} } \partial_{i} \varphi \left(  \textbf{z}\right)\varsigma_{i}\overline{\varsigma_{j}} \geqslant 0$ for  each $ \textbf{z}\in V$, $n\in \mathbb{N}$ and $\varsigma_{1},\cdots, \varsigma_{n}\in \mathbb{C}$. 	\end{itemize}
\end{lemma}

\begin{proof}
	``$(a)\Rightarrow (b)$": By the conclusion $(1)$ of Proposition \ref{dchcthhshzyxwdxzh23616}, it is obvious.
		
\medskip

``$(b)\Rightarrow (c)$": It is obvious.

\medskip

``$(c)\Rightarrow (a)$": For each $\textbf{z}\in V$, there is $\rho>0$ such that $\overline{B_\rho \left(\textbf{z}\right)  } \subset V$.
		For each $\delta>0$ and $\zeta=\left( \zeta_{i}\right)^{\infty }_{i=1} \in \ell^{2}\setminus\{ 0\}$, we fix $r\in \left(0,\min \left\{ \delta ,\rho/\left| \left| \zeta\right|  \right|   \right\} \right) $ and write $\zeta^{(n)}=\left( \zeta_{i}\right)^{n}_{i=1} $. From Definition \ref{dchcthhshdy23615} and $$\left| \left| \textbf{z}+re^{\sqrt{-1}\theta }\zeta^{\left( n\right)  }-\textbf{z}\right|  \right|  =r\left| \left| \zeta^{\left( n\right)  }\right|  \right|  \leqslant r\left| \left| \zeta\right|  \right|  <\rho  ,\q\forall\; n\in \mathbb{N},$$ we have $$\varphi \left( \textbf{z}\right)  \leqslant \frac{1}{2\pi } \int^{2\pi }_{0} \varphi \left( \textbf{z}+re^{\sqrt{-1}\theta }\zeta^{\left( n\right)  }\right)  d\theta, \q\forall\; n\in \mathbb{N}.$$
		For each $\varepsilon>0$, there is $\gamma>0$ such that $B_\gamma\left( \textbf{z}+re^{\sqrt{-1}\theta }\zeta \right)  \subset V$ and
 $$\left| \varphi \left( \textbf{w}\right)  -\varphi \left( \textbf{z}+re^{\sqrt{-1}\theta }\zeta\right)  \right|   <\varepsilon ,\q\forall \;\textbf{w}\in B_\gamma\left( \textbf{z}+re^{\sqrt{-1}\theta }\zeta\right)  .$$
		
		By $\lim\limits_{n\rightarrow \infty } \left| \left| \zeta-\zeta^{\left( n\right)  }\right|  \right|  =0$, there is $N\in \mathbb{N}$ such that
 $$\left| \left| \zeta-\zeta^{\left( n\right)  }\right|  \right|  <\frac{\gamma}{r} ,\q\forall\; n> N.$$
		
		For each $n>N$ and $\theta \in \left[ 0,2\pi \right]  $, we have
 $$
 \left| \left| \left( \textbf{z}+re^{\sqrt{-1}\theta }\zeta\right)  -\left( \textbf{z}+re^{\sqrt{-1}\theta }\zeta^{\left( n\right)  }\right)  \right|  \right|  =r\left| \left| \zeta-\zeta^{\left( n\right)  }\right|  \right|  <\gamma .
 $$
It follows that
		$$
			\begin{aligned}
				&\left| \frac{1}{2\pi } \int^{2\pi }_{0} \varphi \left( \textbf{z}+re^{\sqrt{-1}\theta }\zeta\right)  d\theta -\frac{1}{2\pi } \int^{2\pi }_{0} \varphi \left( \textbf{z}+re^{\sqrt{-1}\theta }\zeta^{\left( n\right)  }\right)  d\theta \right|  \\&	 \leqslant \frac{1}{2\pi } \int^{2\pi }_{0} \left| \varphi \left( \textbf{z}+re^{\sqrt{-1}\theta }\zeta\right)  -\varphi \left( \textbf{z}+re^{\sqrt{-1}\theta }\zeta^{\left( n\right)  }\right)  \right|  d\theta \leqslant \frac{1}{2\pi } \int^{2\pi }_{0} \varepsilon d\theta =\varepsilon  .
			\end{aligned}
		$$
		Hence
 $$\lim_{n\rightarrow \infty } \int^{2\pi }_{0} \varphi \left( \textbf{z}+re^{\sqrt{-1}\theta }\zeta^{\left( n\right)  }\right)  d\theta =\int^{2\pi }_{0} \varphi \left( \textbf{z}+re^{\sqrt{-1}\theta }\zeta\right)  d\theta ,$$ and therefore $\varphi \left( \textbf{z}\right)  \leqslant \frac{1}{2\pi } \int^{2\pi }_{0} \varphi \left( \textbf{z}+re^{\sqrt{-1}\theta }\zeta \right)  d\theta.$
		By Proposition \ref{dchcthhshjbd23617}, we see that $\varphi $ is plurisubharmonic on $V$.
		
\medskip

``$(c)\Leftrightarrow  (d)$": This equivalence follows from a direct calculation and Proposition \ref{cthgh236129}.
	This completes the proof of Lemma \ref{ghqd236130}.
\end{proof}

\section{Exhaustion functions and Lipschitz functions}\label{sec3}

As in \cite[Definition 2.2]{WYZ}, any real-valued function $\eta$ on $V$ is called an exhaustion function on $V$, if
 \begin{equation}\label{gx2401132}
V_{t}\triangleq \{\textbf{z}\in V:\;\eta (\textbf{z})  \leqslant t\} \stackrel{\circ}{\subset} V,\q \forall\;t\in \mathbb{R}.
 \end{equation}
If $\eta$ is a continuous exhaustion function on $V$, then the interior of $V_{t}$ is given by
 \begin{equation}\label{gx2401133}
 V^{o}_{t}\triangleq \big\{\textbf{z}\in V:\; \eta \left( \textbf{z}\right) < t\big\}.
\end{equation}
In the case that we hope to emphasize the exhaustion function on $V$ under consideration, we re-write the above sets $V_{t}$ and $V^{o}_{t}$ respectively as $V_{\eta,t}$ and $V^{o}_{\eta,t}$.

From the proof of \cite[Proposition 2.6]{WYZ}, we have the following result.
\begin{proposition}\label{jhbh23619}
	Suppose $\eta \in C^1_{F}\left( V\right)   $ is  an exhaustion function on $V$. Then the following holds:
	
$(1)$ For each $t\in \mathbb{R}$, there exists $C\left( t\right)  \in \left( 0,+\infty \right)  $ such that
		\begin{equation}\label{230608x1}
\left| \eta \left(\textbf{z}\right)  -\eta \left( \textbf{w}\right)  \right|  \leqslant C\left( t\right)  \left| \left|\textbf{z}-\textbf{w}\right|  \right|,\quad\forall\; \textbf{z},\textbf{w}\in V_{t}^{o};
       \end{equation}

$(2)$ For any $s,t\in \mathbb{R}$ with $s<t$, it holds that $V_{s}\stackrel{\circ}{\subset} V^{o}_{t}$;

$(3)$ For any $S\stackrel{\circ}{\subset} V$, there exists $t\in \mathbb{R}$ such that $S\stackrel{\circ}{\subset} V^{o}_{t}$.
\end{proposition}
\begin{remark}\label{jhbh236110}
	From the proof of \cite[Proposition 2.6]{WYZ}, one can further deduce that, if $\eta$ is a continuous exhaustion function on $V$ and for each $t\in \mathbb{R}$ there exists $C\left( t\right)  \in \left( 0,+\infty \right)  $  such that
	(\ref{230608x1}) holds,
then the conclusions (2) and (3) in Proposition \ref{jhbh23619} still hold true.
\end{remark}

Clearly, (\ref{230608x1}) is a sort of Lipschitz property for the function $\eta$. Generally, denote by $\hbox{\rm Lip$\,$}\left( V\right)$ the set of all complex-valued functions $f$ on $V$ so that for each $S\stackrel{\circ}{\subset} V$, there exists a constant $C(S)>0$ such that
 \begin{equation}\label{230608e1}
\left| f\left( \textbf{z}\right)  -f\left( \textbf{w}\right)  \right|  \leqslant C\left(S\right)  \cdot \left| \left|\textbf{z}-\textbf{w} \right|  \right|  ,\quad\forall\; \textbf{z},\textbf{w}\in S.
 \end{equation}
Any element $f$ in $\hbox{\rm Lip$\,$}\left( V\right)$ is called a Lipschitz function on $V$; as usual, if the the above constant does not depend on $S$, then $f$ is  a globally Lipschitz function on $V$. Clearly, $\hbox{\rm Lip$\,$}\left( V\right)\subset C(V)$.

The following simple result guarantees the existence of Lipschitz exhaustion functions on each $V$.

\begin{lemma}\label{gx230614t1}
There exists a Lipschitz exhaustion function $\eta$ on $V$ with $\inf\limits_{V} \eta >-\infty $ and $\sup\limits_{S} \left| \eta \right|  <\infty $ for each $S\stackrel{\circ}{\subset} V$.
\end{lemma}

\begin{proof}
First of all, since $V$ is a nonempty open set of $\ell^{2}$, we claim that $\partial V=\emptyset$ if and only if $V=\ell^{2}$. Indeed, the ``if" part is obvious. To show the ``only if" part, by $\partial V=\emptyset$, we see that $\ell^{2}=V\cup \partial V\cup (\ell^{2}\setminus V)^o=V\cup (\ell^{2}\setminus V)^o$. By Proposition \ref{collected}, it follows that $\ell^2$ is a connected space, and hence  $V=\ell^{2}$.

Next, recall (\ref{zx230614e1}) for $ d\left( \cdot,\partial V\right)$, we put
$$\eta \left( \textbf{z}\right)  \triangleq \begin{cases}-\ln d\left( \textbf{z},\partial V\right)  +\left| \left| \textbf{z}\right|  \right|^{2}  ,&V\neq \ell^{2}, \\ \left| \left| \textbf{z}\right|  \right|^{2}  ,&V=\ell^{2}. \end{cases}
 $$
Clearly, $\lim\limits_{\textbf{z}\rightarrow \partial V} \big( -\ln d\left( \textbf{z},\partial V\right)  +\left| \left| \textbf{z}\right|  \right|^{2}  \big)  =+\infty  $ for $V\neq \ell^{2}$ and $\lim\limits_{\left| \left| \textbf{z}\right|  \right|\rightarrow \infty} \left| \left| \textbf{z}\right|  \right|^{2}=+\infty $ for $V= \ell^{2}$. By \cite[Lemma 2.4]{WYZ},  we see that $\eta$ is an exhaustion function on $V$. Moreover, for the case that $\partial V\not=\emptyset$, by $\left| d\left( \textbf{z},\partial V\right)  -d\left( \textbf{w},\partial V\right)  \right|  \leqslant \left| \left| \textbf{z}-\textbf{w}\right|  \right|  $ for all $\textbf{z},\textbf{w}\in V$ and the Lagrange mean value theorem, we see that, for each $S\stackrel{\circ}{\subset} V$,
$$\big| -\ln d\left( \textbf{z},\partial V\right)  -\left( -\ln d\left( \textbf{w},\partial V\right)  \right)  \big|  \leqslant \frac{\big| d\left( \textbf{z},\partial V\right)  -d\left( \textbf{w},\partial V\right)  \big|  }{d\left( S,\partial V\right)  } \leqslant \frac{\left| \left| \textbf{z}-\textbf{w}\right|  \right|  }{d\left( S,\partial V\right)  },\quad\forall\;\textbf{z},\textbf{w}\in S.
 $$
Hence $\eta $ is a Lipschitz function. It is easy to see that $\inf\limits_{V} \eta >-\infty $ and $\sup\limits_{S} \left| \eta \right|  <\infty $.
\end{proof}

\begin{proposition}\label{llxhfxgx236116}
It holds that
 $$
 C^{1}\left( V\right)\cap \hbox{\rm Lip$\,$}\left( V\right)\subset  C^{1}_{F}\left( V\right) .
 $$
	\end{proposition}
\begin{proof}
	For any $f\in C^{1}\left( V\right)\cap \hbox{\rm Lip$\,$}\left( V\right)$, by $f=\Re {f}  +\sqrt{-1} \Im {f}   $,  we may simply assume that $f$ is a real-valued function.
For any fixed $S\stackrel{\circ}{\subset} V$, there exists $C(S)>0$ such that (\ref{230608e1}) holds.
For a fixed $\zeta^{0}\in S$, we have
 $$
 \left| f\left( \textbf{z}\right)  \right|  \leqslant C\left( S\right)  \big| \big| \textbf{z}-\zeta^{0} \big|  \big|  +\big| f\big( \zeta^{0} \big)  \big| \leqslant C\left( S\right)  \sup_{\textbf{w}\in S} \big| \big| \textbf{w}-\zeta^{0} \big|  \big|  +\big| f\big( \zeta^{0} \big)  \big| <\infty  ,\q \forall\; \textbf{z}\in S.
 $$
Let $\eta$ be the Lipschitz exhaustion function on $V$ given in Lemma \ref{gx230614t1}. By Proposition \ref{jhbh23619} and Remark \ref{jhbh236110},	one can find $t\in \mathbb{R}$ such that $S\subset V_{t}^{o}$ (Recall (\ref{gx2401132}) for $V_{t}$). Therefore, it suffices to prove that
\begin{equation}\label{yxh236117}
		\sum^{n}_{i=1} \left( \left| \frac{\partial f\left( \textbf{z}\right)  }{\partial x_{j}}   \right|^{2}  +\left| \frac{\partial f\left( \textbf{z}\right)  }{\partial y_{j}} \right|^{2}  \right)  \leqslant C(V_{t}^{o}),\quad  \forall\; n\geqslant 1,\textbf{z}\in V^{o}_{t}.
	\end{equation}
	
	For any fixed $n\in \mathbb{N}$ and $\textbf{z}=\left( \textbf{z}_n,\textbf{z}^n\right) \in V_{t}^{o}$ (Recall (\ref{231128e1}) for $\textbf{z}_n$ and $\textbf{z}^n$), we choose a sufficiently small open ball  $B_{r}(\textbf{z}) \subset V_{t}^{o}$. By (\ref{230608e1}), we have
 $$\left| f\left(\textbf{z}_n,\textbf{z}^n\right)  -f\left(\textbf{w}_n ,\textbf{z}^n\right)  \right|  \leqslant C(V_{t}^{o})\left| \left| \textbf{z}_n-\textbf{w}_n \right|  \right|_{\mathbb{C}^{n} }  ,\q \forall\;\textbf{w}_n\in \mathbb{C}^{n} \hbox{ with } \left( \textbf{w}_n ,\textbf{z}^n\right)  \in B_{r}(\textbf{z}).
  $$
Hence the inequality \eqref{yxh236117} holds.
	
By \eqref{yxh236117}, and letting $n\rightarrow \infty $, we conclude that $f  \in C^{1 }_{F}\left( V\right)  $. This completes the proof of  Proposition \ref{llxhfxgx236116}.
\end{proof}

The following result, motivated by the classical Friedrichs mollifier
technique and the Gross mollifier technique in \cite[Proposition 6 at p. 133 and Proposition 9 at p. 152]{Gro67} (See also \cite[Theorem 3.3.3, p. 57]{Prato}), will play a crucial role in the sequel.

\begin{theorem}\label{pmgjsh236118}
For any $\vartheta \in C^{\infty }_{0}( \ell^{2})  $, $f\in \hbox{\rm Lip$\,$} (\ell^{2}) $ and $\varepsilon>0$, let
 {\rm $$
 f_{\varepsilon }\left(  \textbf{z}\right)  \triangleq \int_{\ell^{2} } f\left(  \textbf{z}-\varepsilon \zeta \right)  \vartheta \left( \zeta \right)  dP\left( \zeta \right)  ,\quad\forall\; \textbf{z}\in \ell^{2} .
 $$}
Then $f_{\varepsilon }\in C^{\infty }_{F^{\infty }}\left( \ell^{2} \right)$; moreover $f_{\varepsilon }\to f$ everywhere in $\ell^2$ as $\varepsilon\to0^+$, provided that $\int_{\ell^{2} } \vartheta \left( \zeta \right)  dP\left( \zeta \right)=1$.
\end{theorem}

\begin{proof}
Since $\vartheta\in C^{\infty }_{0}( \ell^{2})  $, there exists $r>0$ so that $\supp\vartheta \subset B_r$. It is clear that $f\in C(\ell^{2})$. For each $s>0$ and $\textbf{z}\in B_s$, by the definition of $\hbox{\rm Lip$\,$} (\ell^{2}) $ (and particularly recalling (\ref{230608e1}) for $C\left(B_{s+\varepsilon r} \right)$), we have
 \begin{equation}\label{shyhsh23609x1}
 \begin{aligned}
  \left| f_{\varepsilon }\left( \textbf{z}\right)  \right|  &\leqslant \int_{B_{r}} \left| f\left( \textbf{z}-\varepsilon \zeta \right)  \vartheta \left( \zeta \right)  \right|  dP\left( \zeta \right) \leqslant \int_{B_{r}} \left( \left| f\left( \textbf{z}\right)  \right|  +\varepsilon C\left(B_{s+\varepsilon r} \right)  \left| \left| \zeta \right|  \right|  \right)  \cdot \left| \vartheta \left( \zeta \right)  \right|  dP\left( \zeta \right)  \\
  &\leqslant \left( \left| f\left( \textbf{z}\right)  \right|  +\varepsilon C\left( B_{s+\varepsilon r} \right)  r\right)  \int_{B_{r}}  \left| \vartheta \left( \zeta \right)  \right|  dP\left( \zeta \right)  \\
  &\leqslant \left( \left| f\left( 0\right)  \right|  +C\left(B_{s} \right) s +\varepsilon C\left( B_{s+\varepsilon r}  \right)  r\right)  \int_{B_{r}}  \left| \vartheta \left( \zeta \right)  \right|   dP\left( \zeta \right) ,
\end{aligned}
 \end{equation}
which implies that
\begin{equation}\label{shyhsh23609e1}
\sup\limits_{\textbf{z}\in S} \left| f_{\varepsilon }\left( \textbf{z}\right)  \right|  <\infty,\quad \forall\;S\stackrel{\circ}{\subset}\ell^{2} .
\end{equation}

Further, for each $\textbf{z}(=(z_{i} )_{i=1}^\infty), \textbf{w}(=(w_{i} )_{i=1}^\infty)\in \ell^{2} $ and $\alpha,\beta \in \n$, we write $k\triangleq\max \{ | \alpha_{j} |  ,| \beta_{j} |  :j\geqslant 1\}  \in \mathbb{N}_{0}$, and define $\textbf{z}_k$, $\textbf{z}^k$, $\textbf{w}_k$ and $\textbf{w}^k$ similar to \ref{231128e1}). Then,
 \begin{equation}\label{shyhsh23609e2}
 \begin{aligned}
	&\partial_{\alpha } \overline{\partial }_{\beta } f_{\varepsilon }\left( \textbf{z}\right) \\
&=\partial_{\alpha } \overline{\partial }_{\beta } \left( \frac{1}{\left( 2\pi \right)^{k}  a^{2}_{1}a^{2}_{2}\cdots a^{2}_{k}} \int \int_{\mathbb{C}^{k} } f\left( \textbf{z}-\varepsilon \textbf{w} \right)  e^{-\sum\limits^{k}_{i=1} \frac{| w_i|^{2}  }{2a^{2}_{i}} }\vartheta \left( \textbf{w} \right)  dm_{2k}\left( \textbf{w}_k \right)  dP_{k}\left( \textbf{w}^k\right)  \right)\\
&=\partial_{\alpha } \overline{\partial }_{\beta } \left( \frac{1}{\left( 2\pi \varepsilon^{2} \right)^{k}  a^{2}_{1}a^{2}_{2}\cdots a^{2}_{k}} \int  \int_{\mathbb{C}^{k} } f\left( \textbf{w}_k ,\textbf{z}^k-\varepsilon\textbf{w}^k\right)  \vartheta \left( \frac{\textbf{z}_k-\textbf{w}_k }{\varepsilon } ,\textbf{w}^k \right)\right.\\&\qq\qq\qq\qq\qq\qq\qq\qq\times\left.  e^{-\sum\limits^{k}_{i=1} \frac{| z_{i}-w_{i} |^{2}  }{2a^{2}_{i}\varepsilon^{2} } }dm_{2k}\left( \textbf{w}_k \right)   dP_{k}\left( \textbf{w}^k\right)  \right)
   \\&=\frac{1}{\left( 2\pi \varepsilon^{2} \right)^{k}  a^{2}_{1}a^{2}_{2}\cdots a^{2}_{k}} \int \int_{\mathbb{C}^{k} } f\left( \textbf{w}_k ,\textbf{z}^k-\varepsilon \textbf{w}^k \right)  \partial_{\alpha } \overline{\partial }_{\beta } \left( \vartheta \left( \frac{\textbf{z}_k-\textbf{w}_k }{\varepsilon } ,\textbf{w}^k\right)  \right.\\&\qq\qq\qq\qq\qq\qq\qq\qq\times\left. e^{-\sum\limits^{k}_{i=1} \frac{| z_{i}-w_{i} |^{2}  }{2a^{2}_{i}\varepsilon^{2} } }\right)  dm_{2k}\left( \textbf{w}_k \right)  dP_{k}\left( \textbf{w}^k \right)
   \\&=\frac{1}{\left( 2\pi \varepsilon^{2} \right)^{k}  a^{2}_{1}a^{2}_{2}\cdots a^{2}_{k}} \int \int_{\mathbb{C}^{k} } f\left( \textbf{w}_k ,\textbf{z}^k-\varepsilon \textbf{w}^k\right)  \vartheta_{\alpha ,\beta ,\varepsilon } \left( \frac{\textbf{z}_k-\textbf{w}_k }{\varepsilon } ,\textbf{w}^k \right) \\&\qq\qq\qq\qq\qq\qq\qq\qq\times e^{-\sum\limits^{k}_{i=1} \frac{| z_{i}-w_{i} |^{2}  }{2a^{2}_{i}\varepsilon^{2} } }dm_{2k}\left( \textbf{w}_k\right)   dP_{k}\left(\textbf{w}^k\right)
    \\&=\frac{1}{\left( 2\pi \right)^{k}  a^{2}_{1}a^{2}_{2}\cdots a^{2}_{k}} \int \int_{\mathbb{C}^{k} } f\left( \textbf{z}-\varepsilon \textbf{w} \right)  \vartheta_{\alpha ,\beta ,\varepsilon } \left( \textbf{w}\right)  e^{-\sum\limits^{k}_{i=1} \frac{| w_{i} |^{2}  }{2a^{2}_{i}} }dm_{2k}\left( \textbf{w}_k\right)  dP_{k}\left( \textbf{w}^k\right)    \\& =\int_{\ell^{2} } f\left( \textbf{z}-\varepsilon\textbf{w} \right)   \vartheta_{\alpha,\beta,\varepsilon }\left(  \textbf{w}\right)  dP\left( \textbf{w} \right)  ,
\end{aligned}
 \end{equation}
where
$$
\vartheta_{\alpha ,\beta ,\varepsilon } \left(  \textbf{w}\right)  \triangleq \frac{1}{\varepsilon^{|\alpha|+|\beta|} }\partial_{\alpha } \overline{\partial }_{\beta } \left( \vartheta \left(\textbf{w}\right)  e^{-\sum\limits^{k}_{i=1} \frac{|w_{i} |^{2}  }{2a^{2}_{i}} }\right) \cdot e^{\sum\limits^{k}_{i=1} \frac{|w_{i}|^{2}  }{2a^{2}_{i}} }.
 $$
Moreover, for any $\rho>0$ and $\textbf{z},\textbf{z}'\in B_\rho$, we have
 \begin{equation}\label{shyhsh23609e3}
 \begin{aligned}
		&|\partial_{\alpha } \overline{\partial }_{\beta }f_{\varepsilon }\left( \textbf{z}\right)  -\partial_{\alpha } \overline{\partial }_{\beta }f_{\varepsilon }\left( \textbf{z}'\right)  |\leqslant \int_{\ell^{2} } |f\left( \textbf{z}-\varepsilon \textbf{w} \right)  -f\left( \textbf{z}'-\varepsilon \textbf{w} \right)  |\cdot |\vartheta_{\alpha ,\beta ,\varepsilon } \left( \textbf{w} \right)  |dP\left( \textbf{w} \right)  \\&\leqslant C\left( B_{\rho +\varepsilon r} \right)  \int_{\ell^{2} } \left| \left| \textbf{z}-\textbf{z}'\right|  \right|  \cdot |\vartheta_{\alpha ,\beta ,\varepsilon } \left(\textbf{w}\right)  |dP\left(\textbf{w}\right)
	=C\left( B_{ \rho +\varepsilon r}\right)  \int_{\ell^{2} } |\vartheta_{\alpha ,\beta ,\varepsilon } |dP\cdot \left| \left| \textbf{z}-\textbf{z}'\right|  \right| .
\end{aligned}
 \end{equation}
By $\vartheta \in C^{\infty }_{0}\left( \ell^{2} \right)  $, we see that
\begin{equation}\label{shyhsh23609e4}
\vartheta_{\alpha ,\beta ,\varepsilon } \in C^{\infty }_{0}\left( \ell^{2} \right),\quad \supp \vartheta_{\alpha ,\beta ,\varepsilon }\subset \supp \vartheta  .
  \end{equation}
Hence, by (\ref{shyhsh23609e2})--(\ref{shyhsh23609e3}), it follows that $f_{\varepsilon }\in C^{\infty }\left( \ell^{2} \right)  $.

In order to show $f_{\varepsilon }\in C^{\infty }_{F^{\infty }}\left( \ell^{2} \right)  $, it suffices to prove $\partial_{\alpha } \overline{\partial }_{\beta } f_{\varepsilon } \in C^{1}_{F}\left( \ell^{2} \right)  $ for any fixed $\alpha,\beta\in \n$, which follows from Proposition \ref{llxhfxgx236116} and (\ref{shyhsh23609e2})--(\ref{shyhsh23609e4}).

Finally, in the case that $\int_{\ell^{2} } \vartheta \left( \zeta \right)  dP\left( \zeta \right)=1$, similar to the proof of (\ref{shyhsh23609x1}), for each $\varepsilon \in(0,1)$, it is easy to find that
 $$
  \left| f_{\varepsilon }\left( \textbf{z}\right)- f\left( \textbf{z}\right)  \right| \leqslant \int_{B_{r}} \left|\left( f\left( \textbf{z}-\varepsilon \zeta \right)  -f\left( \textbf{z}\right) \right)\vartheta \left( \zeta \right)\right|  dP\left( \zeta \right) \leqslant \varepsilon rC\left( B_{s+r} \right)  \int_{B_{r}}  \left| \vartheta \left( \zeta \right)  \right|   dP\left( \zeta \right),
 $$
which gives $f_{\varepsilon }\left( \textbf{z}\right)\to f\left( \textbf{z}\right)$ as $\varepsilon\to0^+$. This completes the proof of  Theorem \ref{pmgjsh236118}.
\end{proof}

In the sequel, for any $\tau\geqslant 0$, we shall use frequently the following function:
    \begin{equation}\label{shyhsh236113}
	\mathcal{I}_{\tau} \left( t\right)  =\begin{cases}1,&t\in (-\infty ,\tau],
\\ \left( e^{\frac{1}{t-\tau-1} }-1\right)  e^{-\frac{e^{\frac{1}{t-\tau-1} }}{t-\tau} }+1,&t\in \left( \tau,\tau+1\right),
\\ 0,&t\in [\tau+1,+\infty ).\end{cases}
   \end{equation}

   \begin{proposition}\label{pro240130}
	It holds that
    \begin{equation}\label{shyhshdxzh236114}
  \mathcal{I}_{\tau} \in C^{\infty }\left( \mathbb{R} \right),\q 0\leqslant \mathcal{I}_{\tau} \leqslant 1,\quad -K_0\leqslant \mathcal{I}^{\prime }_{\tau} \leqslant 0,\q\forall\;\tau\geqslant 0,
\end{equation}
where $K_0>0$ is a constant, independent of $\tau$.
\end{proposition}
\begin{proof}
Without loss of generality, we consider only the case $\tau=0$. For $x\in \left( 0,1\right)  $, we let $f\left( x\right)  =\left( e^{\frac{1}{x-1} }-1\right)  e^{-\frac{e^{\frac{1}{x-1} }}{x} }$. One can check that
$$
f^{\prime }\left( x\right)  =\frac{e^{\frac{1}{x-1} -\frac{e^{\frac{1}{x-1} }}{x} }}{x^{2}\left( 1-x\right)^{2}  } \cdot \left[ -2x^{2}+x+e^{\frac{1}{x-1} }\left( x^{2}-x+1\right)  -1\right] .
$$
Hence
$$
 f^{\prime }\left( x\right)  \leqslant \frac{e^{\frac{1}{x-1} -\frac{e^{\frac{1}{x-1} }}{x} }}{x^{2}\left( 1-x\right)^{2}  } \cdot \left[ -2x^{2}+x+\left( x^{2}-x+1\right)  -1\right] =-\frac{e^{\frac{1}{x-1} -\frac{e^{\frac{1}{x-1} }}{x} }}{\left( 1-x\right)^{2}  } < 0.
$$
On the other hand, noting that $\lim_{x\to0^+}f^{\prime }\left( x\right) =\lim_{x\to1^-}f^{\prime }\left( x\right) =0$, we may choose
$$K_0=\max_{0<x<1}|f^{\prime }\left( x\right)|.$$
This completes the proof of Proposition \ref{pro240130}.
\end{proof}

Further, in what follows, for all $\tau\geqslant 0$, we set
 \begin{equation}\label{lsh23611912e1}
 I_{\tau}\left(  \textbf{z}\right)  \triangleq\begin{cases}\mathcal{I}_{\tau} \left( \eta \left(  \textbf{z}\right)  \right)  ,& \textbf{z}\in V\\ 0,& \textbf{z}\notin V\end{cases},
 \end{equation}
where $\mathcal{I}_{\tau}$  is given by \eqref{shyhsh236113}, and $\eta$ is the exhaustion function (on $V$) under consideration (which may be changed from one place to another but it can be distinguished from the context).
Also, we choose a function:
 \begin{equation}\label{lsh23611911}
 \Theta \left( \cdot\right)  \in C^{\infty }\left( \mathbb{R} \right)\hbox{ with }0\leqslant \Theta \left( \cdot\right)  \leqslant 1, \Theta \left( t\right)  =1 \hbox{ when }\left| t\right|  \leqslant \frac{1}{4}\hbox{, and }\Theta \left( t\right)  =0\hbox{ when }\left| t\right|  \geqslant 1.
   \end{equation}
 Put
  \begin{equation}\label{lsh23611912}
 \vartheta \left(  \textbf{z}\right)  \triangleq\Theta \left( \left| \left|  \textbf{z}\right|  \right|^{2}  \right),\q \textbf{z}\in \ell^{2}.
 \end{equation}
Since \cite[Lemma 2.1]{WYZ} gives $P\left( \left\{ \textbf{z}\in \ell^{2} :\left| \left| \textbf{z}\right|  \right|  <\frac{1}{4} \right\}  \right)  >0$, we can choose a constant
\begin{equation}\label{240122e9}
c\triangleq \left( \int_{\ell^{2} } \vartheta \left( \zeta \right)  dP\left( \zeta \right)  \right)^{-1}  >0,
\end{equation}
where $\vartheta \left( \cdot\right)$ is given by (\ref{lsh23611912}). By \cite[Lemma 5.2]{WYZ}, we may choose a function
\begin{equation}\label{240122e1}
\psi \in C^{\infty }\left( \mathbb{R} \right)  \hbox{ with }\psi \left( t\right)  =0\hbox{ for all }t\leqslant 0,\hbox{ and }\psi^{\left( j\right)  } \left( t\right)  >0\hbox{ for all }t>0\hbox{ and  }j=0,1,2.
\end{equation}

The following result guarantees the existence of smooth exhaustion functions on $V$.

\begin{theorem}\label{ghqjhshczx23613}
There exists an exhaustion function (on $V$) in the space $C^{\infty }_{F^{\infty}}\left( V\right)  $.
\end{theorem}

\begin{proof}
It suffices to consider the case $V\not=\ell^2$. From the proof of Lemma \ref{gx230614t1}, we may choose $c_0>0$ such that
$$\eta \left( \textbf{z}\right)  \triangleq -\ln d\left( \textbf{z},\partial V\right)  +\left| \left| \textbf{z}\right|  \right|^{2}+c_0  \geqslant 0,\q \textbf{z}\in V.$$ Moreover,  $\eta$ is a continuous exhaustion function on $V$ and for each $t \geqslant 0$, there exists a constant $C(t)>0, $ such that
 \begin{equation}\label{lsh236119}
 \left\{
\begin{array}{ll}
		\left| \eta \left(  \textbf{z}\right)  -\eta \left(  \textbf{w}\right)  \right|  \leqslant C\left( t\right)  \left| \left|  \textbf{z}- \textbf{w}\right|  \right|  ,\\[3mm]
\left| \eta \left(  \textbf{z}\right)  \right|  \leqslant C\left( t\right)  ,
\end{array}\right.\q\forall\;  \textbf{z}, \textbf{w}\in V_{t}.
	\end{equation}
	
	Our goal is to construct a function $\Psi \in C^{\infty }_{F^{\infty}}\left( V\right)  $ with $\eta \left( \cdot\right)  \leqslant \Psi \left( \cdot\right)  $.
	
By the conclusion $(2)$ of Proposition \ref{jhbh23619} and Remark \ref{jhbh236110} and noting \eqref{lsh236119}, one has $V_{j}\stackrel{\circ}{\subset} V^{o}_{j+{1}/{2} }$ and $V_{j+2}\stackrel{\circ}{\subset} V^{o}_{j+3}$ for each $j\in \mathbb{N}$. We may choose $\varepsilon_{j} \in \left( 0,\frac{1}{2C\left( j+{1}/{2} \right)  } \right)$ such that
	\begin{eqnarray}\label{fw236120}
 \overline{B\left(  \textbf{z},\varepsilon_{j} \right)  } \subset V^{o}_{j+{1}/{2} },\q \forall \; \textbf{z}\in V^{o}_{j}
	\end{eqnarray}
and
	\begin{eqnarray}\label{fw236121}
		\overline{B\left(  \textbf{z},\varepsilon_{j} \right)  } \bigcap V^{o}_{j+2}=\emptyset ,\q \forall \; \textbf{z}\in \ell^{2} \setminus V^{o}_{j+3}.
	\end{eqnarray}

Recall (\ref{lsh23611912e1}) for $I_j$. We claim that
\begin{eqnarray}\label{fw2401121}
I_{j+1}(\cdot)\eta(\cdot)\in \hbox{\rm Lip$\,$}(\ell^2).
\end{eqnarray}
Indeed, for each $ \textbf{z}, \textbf{w}\in V_{j+4}^{o}$, using \eqref{lsh236119}, noting $\left| \mathcal{I}^{\prime }_{j+1} \right|  \leqslant K_0$ (given in Proposition \ref{pro240130}) and $\left| I_{j+1}\right|  \leqslant 1$ (by \eqref{shyhshdxzh236114}), and by the Lagrange mean value theorem, we have
$$\begin{aligned}
\left| I_{j+1}\left(  \textbf{z}\right)  \eta \left(  \textbf{z}\right)  -I_{j+1}\left(  \textbf{w}\right)  \eta \left(  \textbf{w}\right)  \right|   &\leqslant \left| I_{j+1}\left( \textbf{z}\right)  \right|  \cdot \left| \eta \left(  \textbf{z}\right)  -\eta \left( \textbf{w}\right)  \right|  +\left| I_{j+1}\left(  \textbf{z}\right)  -I_{j+1}\left(  \textbf{w}\right)  \right|  \cdot \left| \eta \left(  \textbf{w} \right)  \right| \\&\leqslant \left| \eta \left(  \textbf{z} \right)  -\eta \left(  \textbf{w} \right)  \right| +C\left( j+4\right)  \left| I_{j+1}\left(  \textbf{z}\right)  -I_{j+1}\left(  \textbf{w} \right)  \right|\\&   \leqslant \left| \eta \left(  \textbf{z}\right)  -\eta \left(  \textbf{w}\right)  \right|  +K_0C\left( j+4\right)  \left| \eta \left(  \textbf{z}\right)  -\eta \left(  \textbf{w} \right)  \right| \\&=\left[ 1+K_0C\left( j+4\right)  \right]  \cdot \left| \eta \left(  \textbf{z}\right)  -\eta \left(  \textbf{w} \right)  \right|  \\&\leqslant \left[ 1+K_0C\left( j+4\right)  \right]  C\left( j+4\right)  \cdot \left| \left| \textbf{z}- \textbf{w}\right|  \right|    .
 \end{aligned}$$
By $\supp \left( I_{j+1}\eta \right)  \subset V_{j+2} \stackrel{\circ}{\subset} V $, we have
 $$
 \begin{aligned}
 &\left| I_{j+1}\left(  \textbf{z}\right)  \eta \left(  \textbf{z}\right) -I_{j+1}\left(  \textbf{w}\right)  \eta \left(  \textbf{w}\right)  \right|  =0,
 \\&\qq\qq\qq\forall\; ( \textbf{z}, \textbf{w})\in\left(\left( \ell^{2} \setminus V^{o}_{j+4}\right)\times\left( \ell^{2} \setminus V_{j+2}\right)\right)\bigcup
 \left(\left( V^{o}_{j+4} \setminus V_{j+2}\right)\times\left( \ell^{2} \setminus V_{j+4}^{o}\right)\right).
\end{aligned} $$
For each $( \textbf{z}, \textbf{w})\in \left(\left(\ell^{2} \setminus V^{o}_{j+4}\right)\times V_{j+2}\right)\bigcup \left(V_{j+2}\times\left(\ell^{2} \setminus V^{o}_{j+4}\right)\right)$, \eqref{fw236121} gives $\left| \left| \textbf{z}-\textbf{w}\right|  \right|  > \varepsilon_{j} $, and hence
$$\left| I_{j+1}\left( \textbf{z}\right)  \eta \left( \textbf{z}\right)  -I_{j+1}\left( \textbf{w}\right)  \eta \left( \textbf{w}\right)  \right|  =\left| I_{j+1}\left( \textbf{w}\right)  \eta \left( \textbf{w}\right)  \right| \leqslant j+2  < \frac{j+2}{\varepsilon_{j} } \left| \left| \textbf{z}-\textbf{w}\right|  \right|  .$$
Combining the above, it is easy to deduce that $$\sup_{\textbf{z},\textbf{w}\in \ell^{2},\textbf{z}\neq \textbf{w}} \frac{\left| I_{j+1}\left( \textbf{z}\right)  \eta \left( \textbf{z}\right)  -I_{j+1}\left( \textbf{w}\right)  \eta \left( \textbf{w}\right)  \right|  }{\left| \left| \textbf{z}-\textbf{w}\right|  \right|  }<\infty  ,
 $$
which gives (\ref{fw2401121}).

We put (Recall (\ref{240122e9}) and (\ref{lsh23611912}) respectively for $c$ and $\vartheta \left( \cdot\right)$)
$$\eta_{j} \left(  \textbf{z}\right)  \triangleq c\int_{\ell^{2} } I_{j+1}\left(  \textbf{z}-\varepsilon_{j} \zeta \right)  \eta \left(  \textbf{z}-\varepsilon_{j} \zeta \right)  \vartheta \left( \zeta \right)  dP\left( \zeta \right) ,\q\forall\; \textbf{z}\in\ell^2   .
 $$
By $\vartheta \in C^{\infty }_{0}\left( \ell^{2} \right)  $, (\ref{fw2401121}) and Theorem \ref{pmgjsh236118}, we have $ \eta_{j} \in C^{\infty }_{F^{\infty }}\left( \ell^{2}\right)   $. Also,
\eqref{fw236121} gives $\supp  \eta_{j} \subset V_{j+3}$, and hence $\eta_{j} \in C^{\infty }_{0,F^{\infty}}\left( V\right)  $.
	
	By \eqref{fw236120}, for each $ \textbf{z}\in V^{o}_{j}$ and $\zeta \in B_1 $, one have $ \textbf{z}-\varepsilon_{j} \zeta \in V^{o}_{j+{1}/{2} }$. By \eqref{lsh236119} and noting that $\supp \vartheta(\cdot)\subset B_1 $,  we have
 \begin{equation}\label{lsh236122}
		\begin{aligned}
			\eta_{j} \left(  \textbf{z}\right)  &=c\int_{\ell^{2} } \eta \left(  \textbf{z}-\varepsilon_{j} \zeta \right)  \vartheta \left( \zeta \right)  dP\left( \zeta \right) \\& \geqslant c\int_{\ell^{2} } \left( \eta \left(  \textbf{z}\right)  -C\left( j+{1}/{2} \right)  \varepsilon_{j} \left| \left| \zeta \right|  \right|  \right)  \vartheta \left( \zeta \right)  dP\left( \zeta \right)  \geqslant \eta \left(  \textbf{z}\right)  -C\left( j+{1}/{2} \right)  \varepsilon_{j},
		\end{aligned}
	\end{equation}
	and
\begin{equation}\label{lsh236123}
		\begin{aligned}
			\eta_{j} \left( \textbf{z}\right)  &=c\int_{\ell^{2} } \eta \left( \textbf{z}-\varepsilon_{j} \zeta \right)  \vartheta \left( \zeta \right)  dP\left( \zeta \right) \\& \leqslant c\int_{\ell^{2} } \left( \eta \left( \textbf{z}\right)  +C\left( j+{1}/{2} \right)  \varepsilon_{j} \left| \left| \zeta \right|  \right|  \right)  \vartheta \left( \zeta \right)  dP\left( \zeta \right)\leqslant \eta \left( \textbf{z}\right)  +C\left( j+{1}/{2} \right)  \varepsilon_{j} .
		\end{aligned}
	\end{equation}
	Hence, combining \eqref{lsh236122} and  \eqref{lsh236123} with $\varepsilon_{j} \in \left( 0,\frac{1}{2C\left( j+{1}/{2} \right)  } \right)  $, we arrive at
 \begin{eqnarray}
		\eta \left( \textbf{z}\right)  \leqslant \eta_{j} \left( \textbf{z}\right)  +C\left( j+{1}/{2} \right)  \varepsilon_{j} \leqslant \eta \left( \textbf{z}\right)  +2C\left( j+{1}/{2} \right)  \varepsilon_{j} <\eta \left( \textbf{z}\right)  +1,\q\forall\;\textbf{z}\in V^{o}_{j}.\label{lsh236125}
	\end{eqnarray}

Recalling (\ref{240122e1}) for $\psi $, we put
 $$\Psi \left( \textbf{z}\right)  \triangleq\sum^{\infty }_{j=1} \frac{1}{\psi \left( 1\right)  }\sup\limits_{V_{j}} \eta \cdot \psi \left( \eta_{j} \left( \textbf{z}\right)  +C\left( j+{1}/{2} \right)  \varepsilon_{j} +2-j\right)  ,\q\textbf{z}\in V.$$
Note that for each $j,k\in \mathbb{N}, j\geqslant k+3$ and $\textbf{z}\in V^{o}_{k}$, by \eqref{lsh236125}, we have $$\eta_{j} \left( \textbf{z}\right)  +C\left( j+{1}/{2} \right)  \varepsilon_{j} +2-j<\eta \left( \textbf{z}\right)  +1+2-j<k+3-j\leqslant 0,$$ which means that $$\psi \left( \eta_{j} \left( \textbf{z}\right)  +C\left( j+{1}/{2} \right)  \varepsilon_{j} +2-j\right)  =0,\q \forall\;j\geqslant k+3,\textbf{z}\in V^{o}_{k}.$$
	Therefore, $\Psi \left( \textbf{z}\right)  = \sum\limits^{k+2}_{j=1} \frac{1}{\psi \left( 1\right)  } \sup\limits_{V_{j}} \eta\cdot \psi \left( \eta_{j} \left( \textbf{z}\right)  +C\left( j+{1}/{2} \right)  \varepsilon_{j} +2-j\right)  $ for $\textbf{z}\in V^{o}_{k}$. Hence $\Psi $ is well-defined on $V$ and $\Psi \in C^{\infty }_{F^{\infty }}\left( V\right)  $.
	
If $\textbf{z}\in V_{1}^{o}$, then, by \eqref{lsh236125} and $\eta(\textbf{z}) \geqslant 0$, it follows that
 $$\eta_{1} \left( \textbf{z}\right)  +C\left( 1+{1}/{2} \right)  \varepsilon_{1} +2-1\geqslant \eta \left( \textbf{z}\right)  +1\geqslant 1,$$ which means that $$\Psi \left( \textbf{z}\right)  \geqslant \frac{1}{\psi \left( 1\right)  } \sup\limits_{V_{1}} \eta \cdot\psi \left( 1\right)  \geqslant \eta \left( \textbf{z}\right)  .$$
	
	If $k\geqslant2$ and $\textbf{z}\in V_{k}^{o}\setminus V_{k-1}^{o}$,  then \eqref{lsh236125} gives
 $$\eta_{k} \left( \textbf{z}\right)  +C\left( k+{1}/{2} \right)  \varepsilon_{k} +2-k\geqslant k-1+2-k=1,
 $$
which means that $$\Psi \left( \textbf{z}\right)  \geqslant \frac{1}{\psi \left( 1\right)  } \sup\limits_{V_{k}} \eta \cdot\psi \left( 1\right)  \geqslant \eta \left( \textbf{z}\right)  .$$
This completes the proof of Theorem \ref{ghqjhshczx23613}.
\end{proof}

\begin{remark}\label{llx2401131}
In the sequel, to simply the notations, unless otherwise stated, we still denote by $\eta$ the exhaustion function (on $V$, in the space $C^{\infty }_{F^{\infty}}\left( V\right)$) given in Theorem \ref{ghqjhshczx23613}, and write $V_{t}$ and $V_{t}^o$ respectively as that in (\ref{gx2401132}) and (\ref{gx2401133}).
\end{remark}

We provide below an application of Theorem \ref{ghqjhshczx23613}, which strengthens the conclusion of Proposition \ref{llxhfxgx236116}.
\begin{proposition}\label{llx236126}
It holds that
 $$
 C^{1}\left( V\right)\cap \hbox{\rm Lip$\,$}\left( V\right)=  C^{1}_{F}\left( V\right) .
 $$
\end{proposition}

\begin{proof}
By Proposition \ref{llxhfxgx236116}, it suffices to prove $ C^{1}_{F}\left( V\right) \subset  C^{1}\left( V\right)\cap \hbox{\rm Lip$\,$}\left( V\right)$. For this purpose we fix $S\stackrel{\circ}{\subset} V$ and $f\in  C^{1}_{F}\left( V\right)$.

By Theorem  \ref{ghqjhshczx23613} and Remark \ref{llx2401131}, we may choose an exhaustion function $\eta$ on $V$ in the space $C^{\infty }_{F^{\infty}}\left( V\right)$. By the conclusion $(3)$ of Proposition \ref{jhbh23619}, there exists $t\in \mathbb{R}$ such that $S\subset V_{t}^{o}$.
	
	Fix $t'>t$ and a function $\phi \in C^{\infty }\left( \mathbb{R} \right) $ with $0\leqslant \phi \leqslant 1$, $\phi \left( s\right)  =1$ for $s<t$, and $ \phi \left( s\right)  =0$ for $s>t^{\prime }$. Clearly, $\Phi \left( \cdot\right)  \triangleq\phi \left( \eta \left( \cdot\right)  \right)  \in C^{\infty }_{0,F}\left( V\right)( \subset L^{2}(\ell^{2},P) )$. Put
 $$
 g_{k}\left( \textbf{z}_k \right)  =\int \Phi \left( \textbf{z}_k,\textbf{z}^k\right) f \left( \textbf{z}_k,\textbf{z}^k\right)  dP_{k}\left( \textbf{z}^k\right), \q k=1,2,\cdots ,
 $$
where $\textbf{z}_k$ and $\textbf{z}^k$ are defined in (\ref{231128e1}) with $n$ therein replaced by $k$.
	
	By the Cauchy-Schwarz inequality and $\supp(\Phi f) \subset V_{t^{\prime }}$, for any $k\in \mathbb{N}$ and $ \textbf{z}_k\in \mathbb{C}^{k}$, we have
 $$
	\begin{aligned}
		\sum^{k}_{j=1} \left| \partial_{j} g_{k}\left(  \textbf{z}_k\right)  \right|^{2}  &=\sum^{k}_{j=1} \left| \int \partial_{j} \left( \Phi \left(  \textbf{z}_k, \textbf{z}^k\right)  f\left( \textbf{z}_k, \textbf{z}^k\right)  \right)  dP_{k}\left(  \textbf{z}^k\right)  \right|^{2}
\\&\leqslant \sum^{k}_{j=1} \int \left| \partial_{j} \left( \Phi \left(  \textbf{z}_k, \textbf{z}^k\right)  f\left(  \textbf{z}_k, \textbf{z}^k\right)  \right)  \right|^{2}  dP_{k}\left(  \textbf{z}^k\right)   \leqslant \sup_{V_{t^{\prime }}} \sum^{\infty }_{j=1} \left| \partial_{j} \left( \Phi f\right)  \right|^{2}.
	\end{aligned}
	$$
Similarly,
$	\sum^{k}_{j=1} \big| \overline{\partial_{j}} g_{k}\left(  \textbf{z}_k\right)  \big|^{2}   \leqslant \sup_{V_{t^{\prime }}} \sum^{\infty }_{j=1} \big| \overline{\partial_{j} } \left( \Phi f\right)  \big|^{2}$.
Hence, there is a positive number $C$, independent of $k$, such that (For each $\textbf{w}\in \ell^2$, $\textbf{w}_k$ is defined similarly as that in (\ref{231128e1}))
$$\left| g_{k}\left(  \textbf{z}_k\right)  -g_{k}\left(  \textbf{w}_k \right)  \right|  \leqslant C\left| \left|  \textbf{z}_k- \textbf{w}_k\right|  \right|_{\mathbb{C}^{k} }  ,\q  \forall\;  \textbf{z}_k, \textbf{w}_k \in \mathbb{C}^{k} .
 $$
Thanks to Proposition \ref{jwcz236112},  $g_{n}$ converges to $\Phi f$ in the Gauss measure as $n\to\infty$, hence there exists a subsequence $\left\{ g_{n_{k}}\right\}^{\infty }_{k=1}  $ of $\left\{ g_{k}\right\}^{\infty }_{k=1}  $, such that   $\lim\limits_{k\rightarrow \infty } g_{n_{k}}=\Phi f ,\  a.e.\;  P$. To simplify the notations, we still denote this subsequence by $\{g_k\}_{k=1}^\infty$. For almost all $ \textbf{z}, \textbf{w}\in V^{o}_{t}$,
$$
	\begin{aligned}
		\left| f \left(  \textbf{z}\right)  -f \left(  \textbf{w} \right)  \right|  &=\left| \Phi \left(  \textbf{z}\right)  f \left(  \textbf{z}\right)  -\Phi \left(  \textbf{w} \right) f \left(  \textbf{w}\right)  \right|  =\lim_{k\rightarrow \infty } \left| g_{k}\left(\textbf{z}_k\right)  -g_{k}\left( \textbf{w}_k \right)  \right|  \\&\leqslant C\lim_{k\rightarrow \infty } \left| \left| \textbf{z}_k-\textbf{w}_k \right|  \right|_{\mathbb{C}^{k} }  =C\left| \left| \textbf{z}-\textbf{w} \right|  \right| .
	\end{aligned}
	$$
By $ f \left( \cdot\right)  \in C\left( V^{o}_{t}\right)  $, we have
$$\left|f \left( \textbf{z}\right)  -f \left( \textbf{w} \right)  \right|  \leqslant C\left| \left| \textbf{z}-\textbf{w} \right|  \right|  ,\q  \forall\; \textbf{z},\textbf{w} \in V^{o}_{t}. $$
This completes the proof of Proposition \ref{llx236126}.
\end{proof}

\section{Semi-anti-plurisubharmonic functions}\label{section4}

We begin with the following notions.

\begin{definition}\label{bthshdy23611}
	A real-valued function $f\in C(V)$ is called globally semi-anti-plurisubharmonic (with a parameter $C$) on $V$, if for each $\textbf{z}\in V, \zeta\in \ell^{2}$ and $\delta>0$, there is $r\in (0,\delta)$ such that $\big\{ \textbf{z}+\rho e^{\sqrt{-1}\theta }\zeta :$ $\rho \in \left[ 0,r\right]  ,\theta \in \left[ 0,2\pi \right]  \big\}  \subset V$ and
 $$
\frac{1}{2\pi } \int^{2\pi }_{0} \left(f\left( \textbf{z}\right)  -f\left( \textbf{z}+re^{\sqrt{-1}\theta }\zeta \right)  \right)d\theta \geqslant -Cr^{2}\left| \left| \zeta \right|  \right|^{2}.
 $$
Generally, a real-valued function $f\in C(V)$ is called semi-anti-plurisubharmonic on $V$, if  $f|_{S}$ is globally semi-anti-plurisubharmonic on $S$ for each nonempty open set $S\stackrel{\circ}{\subset} V$.
\end{definition}
\begin{example}
	$\left| \left| \textbf{z}\right|  \right|^{2}  = \sum\limits^{\infty }_{i=1} \left|z_{i}\right|^{2}  $ is globally semi-anti-plurisubharmonic on $V$ with a parameter $1$ .
\end{example}
\begin{remark}
	Similar to the above, one can define globally semi-plurisubharmonic functions and semi-plurisubharmonic functions on $V$.
\end{remark}

We have the following characterization on globally semi-anti-plurisubharmonic functions on $V$.

\begin{lemma}\label{bthsh236136}
	A real-valued function $f\in C^{2}(V)$ is globally semi-anti-plurisubharmonic on $V$ with a parameter $C$ if and only if$$\sum_{i,j=1}^n  \partial_{i} \overline{\partial_{j} } f\left( \textbf{z}\right) \varsigma_{i}\overline{\varsigma_{j}} \leqslant C\sum^{n}_{i=1} \left| \varsigma_{i}\right|^{2}  ,\q\forall\; \left(n,\textbf{z}\right)\in \mathbb{N} \times V\hbox{ and }\left(\varsigma_{1},\cdots ,\varsigma_{n}\right)\in \mathbb{C}^n.$$
\end{lemma}
\begin{proof}
	{\bf The ``only if" part}. For $\textbf{z}\in V $, we put
\begin{equation}\label{240121e2}
g\left( \textbf{z}\right)  \triangleq -f\left( \textbf{z}\right)  +C\left| \left| \textbf{z}\right|  \right|^{2}.
\end{equation}
For each $\textbf{z}\in V, \zeta \in \ell^{2}$ and $\delta>0$, there is  $r\in (0,\delta)$ such that $\left\{ \textbf{z}+\rho e^{\sqrt{-1}\theta }\zeta :\rho \in \left[ 0,r\right]  ,\theta \in \left[ 0,2\pi \right]  \right\}$ $ \subset V$ and that
		\begin{equation}\label{lsh236134}
\begin{aligned}
			&\frac{1}{2\pi } \int^{2\pi }_{0} \left(g\left( \textbf{z}\right)  -g\left( \textbf{z}+re^{\sqrt{-1}\theta }\zeta \right)  \right)d\theta \\&=\frac{1}{2\pi } \int^{2\pi }_{0} \left( -f\left( \textbf{z}\right)  +C\left| \left| \textbf{z}\right|  \right|^{2}  +f\left( \textbf{z}+re^{\sqrt{-1}\theta }\zeta \right)  -C\left| \left| \textbf{z}+re^{\sqrt{-1}\theta }\zeta \right|  \right|^{2}  \right)  d\theta \\&=-\frac{1}{2\pi } \int^{2\pi }_{0} \left( f\left( \textbf{z}\right)  -f\left( \textbf{z}+re^{\sqrt{-1}\theta }\zeta \right)  \right)  d\theta +\frac{C}{2\pi } \int^{2\pi }_{0} \left( \left| \left| \textbf{z}\right|  \right|^{2}  -\left| \left| \textbf{z}+re^{\sqrt{-1}\theta }\zeta \right|  \right|^{2}  \right)  d\theta \\&=-\frac{1}{2\pi } \int^{2\pi }_{0} \left( f\left( \textbf{z}\right)  -f\left( \textbf{z}+re^{\sqrt{-1}\theta }\zeta \right)  \right)  d\theta -\frac{C}{2\pi } \int^{2\pi }_{0} r^{2}\left| \left| \zeta \right|  \right|^{2}  d\theta \\&\leqslant Cr^{2}\left| \left| \zeta \right|  \right|^{2}  -\frac{C}{2\pi } \int^{2\pi }_{0} r^{2}\left| \left| \zeta \right|  \right|^{2}  d\theta =0.
		\end{aligned}
	\end{equation}
Hence, by Proposition \ref{dchcthhshjbd23617}, $g$ is plurisubharmonic on $V$.
	
	For each $\left(n,\textbf{z}\right)\in \mathbb{N} \times V$ and $\left(\varsigma_{1},\cdots ,\varsigma_{n}\right)\in \mathbb{C}^n$,  by Lemma \ref{ghqd236130}, we have $\sum_{i,j=1}^n  \partial_{i} \overline{\partial_{j} } g\left( \textbf{z}\right)\varsigma_{i}\overline{\varsigma_{j}} \geqslant 0$, which means  that
 $$
		\begin{aligned}
			&\sum_{i,j=1}^n  \partial_{i} \overline{\partial_{j} } f\left( \textbf{z}\right)\varsigma_{i}\overline{\varsigma_{j}} =\sum_{i,j=1}^n  \partial_{i} \overline{\partial_{j} } \left( C\left| \left| \textbf{z}\right|  \right|^{2}  -g\left( \textbf{z}\right)  \right)\varsigma_{i}\overline{\varsigma_{j}}\\&=C\sum_{i,j=1}^n  \partial_{i} \overline{\partial_{j} } \left( \left| \left| \textbf{z}\right|  \right|^{2}  \right)  \varsigma_{i}\overline{\varsigma_{j}} -\sum_{i,j=1}^n  \partial_{i} \overline{\partial_{j} } g\left( \textbf{z}\right)\varsigma_{i}\overline{\varsigma_{j}} \leqslant C\sum^{n}_{i=1} \left|\varsigma_{i}\right|^{2}  .
		\end{aligned}
	$$

\medskip
	
	{\bf The ``if" part}.	
For the function $g $ given by (\ref{240121e2}), and each $\left(n,\textbf{z}\right)\in \mathbb{N} \times V$ and $\varsigma_{1},\cdots ,\varsigma_{n}\in\mathbb{C}$,  we have
 $$C\sum^{n}_{i=1} \left|\varsigma_{i}\right|^{2}  -\sum_{i,j=1}^n  \partial_{i} \overline{\partial_{j} } g\left( \textbf{z}\right)\varsigma_{i}\overline{\varsigma_{j}} =\sum_{i,j=1}^n  \partial_{i} \overline{\partial_{j} } f\left( \textbf{z}\right)\varsigma_{i}\overline{\varsigma_{j}} \leqslant C\sum^{n}_{i=1} \left|\varsigma_{i}\right|^{2}, $$ which means that $$\sum_{i,j=1}^n  \partial_{i} \overline{\partial_{j} } g\left( \textbf{z}\right)\varsigma_{i}\overline{\varsigma_{j}} \geqslant 0.$$
	By Lemma \ref{ghqd236130}, $g$ is plurisubharmonic on $V$.
	
	By Proposition \ref{dchcthhshjbd23617}, for each $\textbf{z}\in V, \zeta\in \ell^{2}$ and $\delta>0$, there exists $r\in (0,\delta)$ such that
 $$\frac{1}{2\pi } \int^{2\pi }_{0} \left(g\left( \textbf{z}\right)  -g\left( \textbf{z}+re^{\sqrt{-1}\theta }\zeta \right)  \right)d\theta \leqslant 0.
 $$
Proceeding as in \eqref{lsh236134}, one deduces that $$ \frac{1}{2\pi } \int^{2\pi }_{0} \left(f\left( \textbf{z}\right)  -f\left( \textbf{z}+re^{\sqrt{-1}\theta }\zeta \right)  \right)d\theta \geqslant -Cr^{2}\left| \left| \zeta \right|  \right|^{2}.
 $$
This completes the proof of Lemma \ref{bthsh236136}.
\end{proof}

In order to construct some useful semi-anti-plurisubharmonic functions on $V$, as in \cite[p. 30 and p. 347]{Prato}, we denote by $UC_{b}\left( \ell^{2} \right)$ the set of all complex-valued, bounded and uniformly continuous functions on $\ell^{2}$, and for any $f\in UC_{b}\left( \ell^{2} \right)  $, we set
 $$w_{f}\left( t\right)\triangleq \sup \left\{ \left| f\left( \textbf{z}\right)  -f\left(  \textbf{w} \right)  \right|  :\textbf{z}, \textbf{w} \in \ell^2\hbox{ with }\left| \left| \textbf{z}- \textbf{w}  \right|  \right|  \leqslant t\right\}  ,\q\forall\; t\geqslant 0.
  $$
  By \cite[Proposition, C.1.1, p. 347]{Prato}, $w_{f}\in C[0,+\infty )$, and for all $t\geqslant s\geqslant 0$,
  $$w_{f}\left( t\right)  \leqslant w_{f}\left( s\right)  ,\q w_{f}\left( t+s\right)  \leqslant w_{f}\left( t\right)  +w_{f}\left( s\right).$$
For any real-valued, bounded function $g$ on $\ell^2$, using the Lasry-Lions regularization technique developed in \cite{LaL86} (stimulated by the inf-sup-convolution formulas, or the Lax-Oleinik
formula for solutions to the classical Hamilton-Jacobi equations),  we set
 $$\left( {\cal U}_{t}g\right)  \left( \textbf{z}\right)  \triangleq \inf_{\zeta \in \ell^{2} } \left\{ g\left( \zeta \right)  +\frac{\left| \left| \textbf{z}-\zeta \right|  \right|^{2}  }{2t} \right\},\q\forall\;\left(\textbf{z},t\right)\in \ell^{2}\times (0,\infty).$$

The following result and its proof are some small modification of the related results from \cite{LaL86} and \cite[Appendix C, p. 347--351]{Prato}.
\begin{lemma}\label{pghcz236131}
	For any real-valued function $f\in UC_{b}\left( \ell^{2} \right)  $, it holds that
\begin{itemize}
		\item[$(a)$] ${\cal U}_{t}f\in UC_{b}\left( \ell^{2} \right)  $, and $\left({\cal U}_{t}f\right)(\cdot)-\frac{\left| \left| \cdot\right|  \right|^{2}  }{2t} $ is anti-plurisubharmonic on $\ell^{2}$;
		\item[$(b)$]
$$f-w_{f}\left( 2\sqrt{t\sup_{\ell^{2} } \left| f\right|  } \right)  \leqslant {\cal U}_{t}f\leqslant f;$$
		\item[$(c)$]
\begin{equation}\label{240121e1}\frac{\left| \left( {\cal U}_{t}f\right)  \left( \textbf{z}\right)  -\left( {\cal U}_{t}f\right)  \left( \zeta \right)  \right|  }{\left| \left| \textbf{z}-\zeta \right|  \right|  } \leqslant \frac{1}{2t} \left(4 \sqrt{t\sup_{\ell^{2} } \left| f\right|  } +\left| \left| \textbf{z}-\zeta \right|  \right|  \right),\q\forall\; \textbf{z},\zeta\in \ell^2\hbox{ with }\textbf{z}\neq \zeta   .
\end{equation}
	\end{itemize}
\end{lemma}

\begin{proof}
	$(a)$	By the definition of ${\cal U}_{t}f$, it is clear that $\inf\limits_{\ell^{2} } f\leqslant {\cal U}_{t}f\leqslant f\leqslant \sup\limits_{\ell^{2} } f$. For each $\textbf{z} ,\textbf{w} \in \ell^{2}$ and $t>0 $, we have
		$$
			\begin{aligned}
		{\cal U}_{t}f \left( \textbf{z} \right)  -{\cal U}_{t}f \left( \textbf{w}\right)  &=\inf_{\xi \in \ell^{2} } \left\{ f\left( \xi \right)  +\frac{\left| \left| \textbf{z}-\xi \right|  \right|^{2}  }{2t} \right\}  -\inf_{\zeta \in \ell^{2} } \left\{ f\left( \zeta \right)  +\frac{\left| \left| \textbf{w}-\zeta \right|  \right|^{2}  }{2t} \right\}   \\&=\inf_{\xi \in \ell^{2} } \left\{ f\left( \xi \right)  +\frac{\left| \left| \textbf{z}-\xi \right|  \right|^{2}  }{2t} \right\}  +\sup_{\zeta \in \ell^{2} } \left\{ -f\left( \zeta \right)  -\frac{\left| \left| \textbf{w}-\zeta \right|  \right|^{2}  }{2t} \right\}   \\&=\sup_{\zeta \in \ell^{2} } \left\{ \inf_{\xi \in \ell^{2} } \left\{ f\left( \xi \right)  +\frac{\left| \left| \textbf{z}-\xi \right|  \right|^{2}  }{2t} \right\}  -f\left( \zeta \right)  -\frac{\left| \left| \textbf{w}-\zeta \right|  \right|^{2}  }{2t} \right\}   \\&=\sup_{\zeta \in \ell^{2} } \left\{ \inf_{\xi \in \ell^{2} } \left\{ f\left( \xi \right)  +\frac{\left| \left| \textbf{z}-\xi \right|  \right|^{2}  }{2t} -f\left( \zeta \right)  -\frac{\left| \left| \textbf{w}-\zeta \right|  \right|^{2}  }{2t} \right\}  \right\}  \\&=\sup_{\zeta \in \ell^{2} } \left\{ \inf_{\xi \in \ell^{2} } \left\{ f\left( \textbf{z}-\xi \right)  +\frac{\left| \left| \xi \right|  \right|^{2}  }{2t} -f\left( \textbf{w}-\zeta \right)  -\frac{\left| \left| \zeta \right|  \right|^{2}  }{2t} \right\}  \right\}  \\&=\sup_{\zeta \in \ell^{2} } \left\{ \inf_{\xi \in \ell^{2} } \left\{ f\left( \textbf{z}-\xi \right)  -f\left( \textbf{w}-\zeta \right)  -\frac{\left| \left| \zeta \right|  \right|^{2}  -\left| \left| \xi \right|  \right|^{2}  }{2t} \right\}  \right\}
   \\&\leqslant \sup_{\zeta \in \ell^{2} } \left\{ \inf_{\xi \in \ell^{2} } \left\{ \left| f\left( \textbf{z}-\xi \right)  -f\left( \textbf{w}-\zeta \right)  \right|  -\frac{\left| \left| \zeta \right|  \right|^{2}  -\left| \left| \xi \right|  \right|^{2}  }{2t} \right\}  \right\}   \\&\leqslant \sup_{\zeta \in \ell^{2} } \left\{ \inf_{\xi \in \ell^{2} } \left\{ w_{f}\left( \left| \left| \textbf{z}-\xi -\textbf{w}+\zeta \right|  \right|  \right)  -\frac{\left| \left| \zeta \right|  \right|^{2}  -\left| \left| \xi \right|  \right|^{2}  }{2t} \right\}  \right\}  \\&\leqslant \sup_{\zeta \in \ell^{2} } \left\{ \inf_{\xi \in \ell^{2} } \left\{ w_{f}\left( \left| \left| \textbf{z}-\textbf{w}\right|  \right|  \right)  +w_{f}\left( \left| \left| \xi -\zeta \right|  \right|  \right)  -\frac{\left| \left| \zeta \right|  \right|^{2}  -\left| \left| \xi \right|  \right|^{2}  }{2t} \right\}  \right\}  \\& \leqslant \sup_{\zeta \in \ell^{2} } \left\{ w_{f}\left( \left| \left| \textbf{z}-\textbf{w}\right|  \right|  \right)  \right\}  = w_{f}\left( \left| \left| \textbf{z}-\textbf{w}\right|  \right|  \right).
			\end{aligned}
		$$
Symmetrically, ${\cal U}_{t}f \left( \textbf{w}\right)  -{\cal U}_{t}f \left( \textbf{z}\right)  \leqslant  w_{f}\left( \left| \left| \textbf{z}-\textbf{w}\right|  \right|  \right)$. Hence $\left| {\cal U}_{t}f \left( \textbf{w}\right)  -{\cal U}_{t}f \left( \textbf{z}\right)  \right|  \leqslant  w_{f}\left( \left| \left| \textbf{z}-\textbf{w}\right|  \right|  \right)$, which indicates that ${\cal U}_{t}f\in UC_{b}\left( \ell^{2} \right)  $.
		
	On the other hand,	
 $$
			\begin{aligned}
				{\cal U}_{t}f \left( \textbf{z}\right)  -\frac{\left| \left| \textbf{z}\right|  \right|^{2}  }{2t} &=\inf_{\zeta \in \ell^{2} } \left\{ f \left( \zeta \right)  +\frac{\left| \left| \textbf{z}-\zeta \right|  \right|^{2}  }{2t} -\frac{\left| \left| \textbf{z}\right|  \right|^{2}  }{2t} \right\}   =\inf_{\zeta \in \ell^{2} } \left\{ f \left( \zeta \right)  -\frac{\Re \left( 2\textbf{z}-\zeta ,\zeta \right)  }{2t} \right\}   .
			\end{aligned}
		$$
		Therefore, for any $a,b,\xi\in \ell^{2}$, by
 $$
  \frac{1}{2t} \Re \left( 2a-\xi,\xi\right)  =\frac{1}{2\pi } \int^{2\pi }_{0} \frac{1}{2t} \Re \left( 2a+2be^{\sqrt{-1}\theta }-\xi,\xi\right)  d\theta ,$$ one has $$
			\begin{aligned}
				\inf_{\zeta \in \ell^{2} } \left\{f\left( \zeta \right)  -\frac{1}{2t} \Re \left( 2a-\zeta ,\zeta \right)  \right\}    &=\inf_{\zeta \in \ell^{2} } \left\{ f\left( \zeta \right)  -\frac{1}{2\pi } \int^{2\pi }_{0} \frac{1}{2t} \Re \left( 2a+2be^{\sqrt{-1}\theta }-\zeta ,\zeta \right)  d\theta \right\}  \\&=\inf_{\zeta \in \ell^{2} } \left\{ \frac{1}{2\pi } \int^{2\pi }_{0} [f\left( \zeta \right)  -\frac{1}{2t} \Re \left( 2a+2be^{\sqrt{-1}\theta }-\zeta ,\zeta \right)  ]d\theta \right\}   \\&\geqslant \frac{1}{2\pi } \int^{2\pi }_{0} \inf_{\zeta \in \ell^{2} } \left\{f\left( \zeta \right)  -\frac{1}{2t} \Re \left( 2a+2be^{\sqrt{-1}\theta }-\zeta ,\zeta \right) \right\}  d\theta  ,
			\end{aligned}
		$$
		which means that $\left({\cal U}_{t}f\right)(\cdot)-\frac{\left| \left| \cdot\right|  \right|^{2}  }{2t} $ is anti-plurisubharmonic on $\ell^{2}$.
	
\medskip
	
$(b)$ By the definition of ${\cal U}_{t}f$, for each $\textbf{z}\in\ell^{2}, \varepsilon>0$ and $t>0$, there exists $y_{\varepsilon }\in \ell^{2}$ such that $${\cal U}_{t}f \left( \textbf{z}\right)  +\varepsilon \geqslant f\left( y_{\varepsilon }\right)  +\frac{\left| \left| \textbf{z}-y_{\varepsilon }\right|  \right|^{2}  }{2t} ,$$hence $$
			\begin{aligned}
				0\leqslant f \left( \textbf{z}\right)  -{\cal U}_{t}f\left( \textbf{z}\right)  &\leqslant f\left( \textbf{z}\right)  -f \left( y_{\varepsilon}\right)  -\frac{\left| \left| \textbf{z}-y_{\varepsilon }\right|  \right|^{2}  }{2t} +\varepsilon \leqslant 2\sup_{\zeta \in \ell^{2}} \left| f \left( \zeta \right)  \right|  -\frac{\left| \left| \textbf{z}-y_{\varepsilon }\right|  \right|^{2}  }{2t} +\varepsilon  ,
			\end{aligned}
		$$
		which means that $$\left| \left| \textbf{z}-y_{\varepsilon }\right|  \right|  \leqslant \sqrt{4t\sup_{\zeta\in \ell^{2}} \left| f \left( \zeta\right)  \right|  +2t\varepsilon } .
 $$
Therefore,
 $$
			\begin{aligned}
				0\leqslant f \left( \textbf{z}\right)  -{\cal U}_{t}f \left( \textbf{z}\right)  &\leqslant f \left( \textbf{z}\right)  -f \left( y_{\varepsilon}\right)  -\frac{\left| \left| \textbf{z}-y_{\varepsilon }\right|  \right|^{2}  }{2t} +\varepsilon \\&\leqslant w_{f }\left( \left| \left| \textbf{z}-y_{\varepsilon }\right|  \right|  \right)  -\frac{\left| \left| \textbf{z}-y_{\varepsilon }\right|  \right|^{2}  }{2t} +\varepsilon\leqslant w_{f}\left( \sqrt{4t\sup_{\zeta \in \ell^{2}} \left| f \left( \zeta\right)  \right|  +2t\varepsilon } \right)  +\varepsilon  .
			\end{aligned}
		$$
		Let $\varepsilon\rightarrow 0^{+}$, we  get $$f-w_{f}\left( 2\sqrt{t\sup_{\ell^{2} } \left| f\right|  } \right)  \leqslant {\cal U}_{t}f\leqslant f.$$

\medskip

		$(c)$ By the definition of ${\cal U}_{t}f$, for each $\varepsilon,t>0$ and $\zeta\in \ell^{2}$, there exists $w_{\varepsilon }\in \ell^{2}$ such that $${\cal U}_{t}f \left(\zeta\right)  +\varepsilon \geqslant f \left( w_{\varepsilon }\right)  +\frac{\left| \left| \zeta-w_{\varepsilon }\right|  \right|^{2}  }{2t} ,$$
 hence,
 \begin{eqnarray}
			\frac{\left| \left| \zeta-w_{\varepsilon }\right|  \right|^{2}  }{2t} \leqslant 2\sup_{ \ell^{2}} \left| f  \right|  +\varepsilon ,\q {\cal U}_{t}f\left( \textbf{z}\right)  \leqslant f\left( w_{\varepsilon }\right)  +\frac{\left| \left| \textbf{z}-w_{\varepsilon }\right|  \right|^{2}  }{2t},\q\forall\;\textbf{z}\in \ell^{2} ,\label{lshxy236133}
		\end{eqnarray}
		which means that  $$
			\begin{aligned}
				{\cal U}_{t}f\left( \textbf{z}\right)  -{\cal U}_{t}f \left( \zeta\right)  &\leqslant \frac{\left| \left| \textbf{z}-w_{\varepsilon }\right|  \right|^{2}  }{2t} -\frac{\left| \left| \zeta-w_{\varepsilon }\right|  \right|^{2}  }{2t} +\varepsilon
\\&=\frac{\left| \left| \textbf{z}-\zeta\right|  \right|^{2}  +2\Re \left( \textbf{z}-\zeta,\zeta-w_{\varepsilon }\right)  }{2t} +\varepsilon \leqslant \frac{\left| \left| \textbf{z}-\zeta \right|  \right|^{2}  +2\left| \left| \zeta -\textbf{z}\right|  \right|  \cdot \left| \left| \zeta -w_{\varepsilon }\right|  \right|  }{2t} +\varepsilon \\&\leqslant \frac{\left| \left|\textbf{z}-\zeta\right|  \right|^{2}  +2\left| \left|\textbf{z}-\zeta\right|  \right|  \left(4t\sup\limits_{\ell^{2}} \left| f  \right|  +2t\varepsilon \right)^{1/2} }{2t} +\varepsilon,
		\end{aligned}$$
		where the last inequality follows from \eqref{lshxy236133}.
		Let $\varepsilon\rightarrow 0^{+}$ in the above, we have
 $${\cal U}_{t}f \left( \textbf{z}\right)  -{\cal U}_{t}f \left( \zeta\right)  \leqslant \frac{\left| \left| \textbf{z}-\zeta\right|  \right|^{2}  +4\left| \left| \textbf{z}-\zeta\right|  \right|  \left(t\sup\limits_{\ell^{2}} \left| f  \right|  \right)^{1/2} }{2t} ,$$Symmetrically, $${\cal U}_{t}f \left( \zeta\right)  -{\cal U}_{t}f \left( \textbf{z}\right)  \leqslant \frac{\left| \left| \textbf{z}-\zeta\right|  \right|^{2}  +4\left| \left| \textbf{z}-\zeta\right|  \right|  \left(t\sup\limits_{\ell^{2}} \left| f  \right|  \right)^{1/2} }{2t} .$$ Hence, we obtain (\ref{240121e1}).	
		This completes the proof of Lemma \ref{pghcz236131}.
\end{proof}

\begin{remark}\label{llxx236132}
The conclusion (c) of Lemma \ref{pghcz236131} indicates that ${\cal U}_{t}f\in \hbox{\rm Lip$\,$}\left(\ell^{2}\right)$.
\end{remark}

%

Now, we are able to give below a key result for overcoming the technical difficulty in infinite dimensional spaces (See \eqref{lsh112361}).
\begin{theorem}\label{shtqj23612}
	There is a semi-anti-plurisubharmonic function $\Psi$ on $V$ such that
   \begin{equation}\label{240121e6}
   \left\{\begin{array}{ll}
    \Psi(\textbf{z}) > 0,\q\forall\; \textbf{z}\in V,\\[3mm]
    \Psi\hbox{ is an exhaustion function on }V,\\[3mm]
    \sup\limits_{\textbf{z},\zeta \in S,\,\textbf{z}\neq \zeta } \left\{\Psi \left( \textbf{z}\right) +\frac{\left| \Psi \left( \textbf{z}\right)  -\Psi \left( \zeta \right)  \right|  }{\left| \left| \textbf{z}-\zeta \right|  \right|  } \right\}<\infty  ,\q\forall \;S \stackrel{\circ}{\subset} V.
	\end{array}\right.
\end{equation}
\end{theorem}

\begin{proof}	
	If $V\neq \ell^{2}$, we let $\eta \left( \textbf{z}\right)  \triangleq \left| \left| \textbf{z}\right|  \right|^{2}_{\ell^{2} }  -\ln d\left( \textbf{z},\partial V\right)+c_0  $, where $c_0$ is a positive number so that $\eta \left( \textbf{z}\right)  >0$ for all $\textbf{z}\in V$. If $V= \ell^{2}$, we let $\eta \left( \textbf{z}\right)  \triangleq 1+\left| \left| \textbf{z}\right|  \right|^{2}_{\ell^{2} } $.
	
We put $V_t$ as that in (\ref{gx2401132}).	By the definition of $d\left( \textbf{z},\partial V\right)  $ and $\left| d\left( \textbf{z},\partial V\right)  -d\left( \zeta ,\partial V\right)  \right|  \leqslant \left| \left| \textbf{z}-\zeta\right|  \right|$ for each $\textbf{z},\zeta \in V$, we have $V_t \stackrel{\circ}{\subset} V$ for all $t\in \mathbb{R}$, and
		 $$\sup_{\textbf{z},\,\zeta \in S,\textbf{z}\neq \zeta } \left\{\eta \left( \textbf{z}\right)+\frac{\left| \eta \left( \textbf{z}\right)  -\eta \left( \zeta \right)  \right|  }{\left| \left| \textbf{z}-\zeta \right|  \right|  } \right\}<\infty  ,\q\forall\; S \stackrel{\circ}{\subset} V.$$

	For each $j\in \mathbb{N}$, we can find a constant $C(j)>0$ so that
$$\left| \eta \left( \textbf{z}\right)  -\eta \left( \zeta \right)  \right|  \leqslant C\left( j\right)  \left| \left| \textbf{z}-\zeta \right|  \right|  ,\q\left| \eta \left( \textbf{z}\right)  \right|  \leqslant C\left( j\right)  ,\q\forall \;\textbf{z},\zeta \in V_{j},$$
	and choose the function $I_{j}\left( \cdot\right)$ on $\ell^2$ as that in (\ref{lsh23611912e1}).
It is clear that $\supp (I_{j}\eta) \stackrel{\circ}{\subset} V$ and $I_{j}\eta \in UC_{b}\left( \ell^{2} \right)  $. 	By Lemma \ref{pghcz236131}, we have ${\cal U}_{t}\left( I_{j}\eta \right)  \in UC_{b}\left( \ell^{2} \right)  $ for each $t>0$ and  $$I_{j}\eta -w_{I_{j}\eta }\left( 2\sqrt{t\sup_{\ell^{2} } \left| I_{j}\eta \right|  } \right)  \leqslant {\cal U}_{t}\left( I_{j}\eta \right)  \leqslant I_{j}\eta .$$
 Since $w_{I_{j}\eta }(\cdot)$ is a continuous function,  we may find $t_{j}>0$ so that
 \begin{eqnarray}
		I_{j}\eta -1\leqslant {\cal U}_{t_{j}}\left( I_{j}\eta \right)  \leqslant I_{j}\eta \leqslant \eta.\label{lshshy236140}
	\end{eqnarray}
We choose the function $\psi$ as that in (\ref{240122e1}). For each $k\in \mathbb{N} $ and $\textbf{z}\in V$, we let
$$\Psi_{k} \left( \textbf{z}\right)  \triangleq \sum^{k}_{j=1} \frac{k}{\psi \left( 1\right)  } \psi \left( {\cal U}_{t_{j}}\left( I_{j}\eta \right)\left( \textbf{z}\right)    +3-j\right)  ,\q\Psi \left( \textbf{z}\right)  \triangleq \lim_{k\rightarrow \infty } \Psi_{k} \left( \textbf{z}\right).$$
	Since \eqref{lshshy236140}, \eqref{240122e1} and
 $${\cal U}_{t_{k+\ell }}\left( I_{k+\ell }\eta \right)  \left( \textbf{z}\right)  +3-k-\ell \leqslant \eta \left( \textbf{z}\right)  +3-k-\ell <3-\ell \leqslant 0,\q  \forall \;\textbf{z}\in V^{o}_{k},\;\ell \geqslant 3,
 $$
we have $\Psi_{k+\ell} \left( \textbf{z}\right)  =\Psi_{k+2} \left( \textbf{z}\right)  $ for all $\textbf{z}\in V^{o}_{k}$ and $\ell\geqslant 2$, and therefore $\Psi(\cdot)$ is well-defined.
	
	In what follows, we shall prove that the above $\Psi(\cdot)$ is the desired function.
	
	By \eqref{lshshy236140} and \eqref{240122e1}, for any $k\in \mathbb{N}$ and $\textbf{z}\in V_{k}^{o}$, we have $$\Psi \left( \textbf{z}\right)  =\Psi_{k+2} \left( \textbf{z}\right)  =\sum^{k+2}_{j=1} \frac{k}{\psi \left( 1\right)  } \psi \left({\cal U}_{t_{j}}\left( I_{j}\eta \right)  \left( \textbf{z}\right)  +3-j\right)  \leqslant \sum^{k+2}_{j=1} \frac{k}{\psi \left( 1\right)  } \psi \left( \sup_{V_{j+1}} \eta +3-j\right)  .$$
	For each $\textbf{z}\in V_{i}\setminus V_{i-1},i=1,2,\cdots,k$, since \eqref{lshshy236140} and $i-1<\eta \left( \textbf{z}\right)  \leqslant i$,  we have $$\Psi_{k} \left( \textbf{z}\right)  =\sum^{k}_{j=1} \frac{k}{\psi \left( 1\right)  } \psi \left({\cal U}_{t_{j}}\left( I_{j}\eta \right)  +3-j\right)  \geqslant \frac{i}{\psi \left( 1\right)  } \psi \left( \eta +2-i\right)  \geqslant i\geqslant \eta(\textbf{z}) .$$

	We  claim that $\psi \left({\cal U}_{t_{j}}\left( I_{j}\eta \right)  +3-j\right) ,j=1,2,\cdots $ is a globally semi-anti-plurisubharmonic and Lipschitz function on $V^{o}_{k}$ for each $k\in \mathbb{N}$.
	
Write
\begin{equation}\label{240122e7}
c_j^1=\sup \left\{ \psi^{\prime }(s) :\;2-j \leqslant s \leqslant\sup\limits_{V_{j+1}} \eta +3-j\right\},\q c^2_j=\sup\left\{\psi^{\prime \prime } \left( s \right):\; 2-j\leqslant s\leqslant \sup\limits_{V_{j+1}} \eta +3-j \right\}.
\end{equation}	
 Since
\begin{eqnarray}\label{lsh236135}
		-1\leqslant {\cal U}_{t_{j}}\left( I_{j}\eta \right)  \leqslant I_{j}\eta \leqslant \sup_{V_{j+1}} \eta,
	\end{eqnarray}
by the Lagrange mean value theorem, (\ref{240122e7}) and the conclusion (c) of Lemma \ref{pghcz236131}, for any $\textbf{z},\zeta\in V_{k}^{o}$, we have
 $$
		\begin{aligned}
			&\left| \psi \left({\cal U}_{t_{j}}\left( I_{j}\eta \right)  \left( \textbf{z}\right)  +3-j\right)  -\psi \left({\cal U}_{t_{j}}\left( I_{j}\eta \right)  \left( \zeta \right)  +3-j\right)  \right| \\&\leqslant c_j^1  \cdot \left| {\cal U}_{t_{j}}\left( I_{j}\eta \right)  \left( \textbf{z}\right)  -{\cal U}_{t_{j}}\left( I_{j}\eta \right)  \left( \zeta \right)  \right|  \leqslant c_j^1\cdot \left| \left| \textbf{z}-\zeta \right|  \right|  \cdot \frac{1}{2t_{j}} \cdot \left( 4\sqrt{t_{j}\sup_{\ell^{2} } \left| I_{j}\eta \right|  } +\left| \left| \textbf{z}-\zeta \right|  \right|  \right)  \\&\leqslant c_j^1\cdot \left| \left| \textbf{z}-\zeta \right|  \right|  \cdot \frac{1}{t_{j}} \cdot \left( 2\sqrt{t_{j}\sup_{\ell^{2} } \left| I_{j}\eta \right|  } +\sup_{\xi\in V^{o}_{k}} \left| \left| \xi\right|  \right|  \right)  ,\label{fz236144}
		\end{aligned}
	$$
	which means that $\psi \left({\cal U}_{t_{j}}\left( I_{j}\eta \right)  +3-j\right)$ is a globally Lipschitz function on $V^{o}_{k}$.
	
	
By the conclusion (a) in Lemma \ref{pghcz236131}, ${\cal U}_{t_{j}}\left( I_{j}\eta \right) $ is globally semi-anti-plurisubharmonic on $\ell^{2}$, therefore so is
 $$\Xi_{j}\triangleq {\cal U}_{t_{j}}\left( I_{j}\eta \right)  +3-j,\q\forall\;j\in \mathbb{N}.$$
 Hence, there is $C_{j}^{\prime}>0$ so that for each $k\in \mathbb{N}$, $\textbf{z}\in V_{k}^{o}$, $\zeta\in \ell^{2}$ and $\delta >0$, one can choose $r\in (0,\delta)$ such that  $\textbf{z}+\rho \zeta e^{\sqrt{-1}\theta}\in V^{o}_{k}$ for all $\rho \in \left[ 0,r\right]  $ and $\theta \in \left[ 0,2\pi \right]  $, and $$\frac{1}{2\pi } \int^{2\pi }_{0} \left( \Xi_{j} \left( \textbf{z}\right)  -\Xi_{j} \left( \textbf{z}+re^{\sqrt{-1}\theta }\zeta \right)  \right)  d\theta \geqslant -C^{\prime }_{j}r^{2}\left| \left| \zeta \right|  \right|^{2}  .
  $$
By the Taylor mean value theorem, (\ref{240122e7}) and \eqref{lsh236135}, it follows that
  $$
  \begin{aligned}
			&\psi \left(\Xi_{j} \left( \textbf{z}+re^{\sqrt{-1}\theta }\zeta \right)  \right)-\psi \left(\Xi_{j} \left( \textbf{z}\right)  \right)\\
  &\leqslant \psi^{\prime } \left(\Xi_{j} \left( \textbf{z}\right)  \right)  \left(\Xi_{j} \left( \textbf{z}+re^{\sqrt{-1}\theta }\zeta \right)  -\Xi_{j} \left( \textbf{z}\right)  \right)  +\frac{c^2_j }{2} \left(\Xi_{j} \left( \textbf{z}+re^{\sqrt{-1}\theta }\zeta \right)  -\Xi_{j} \left( \textbf{z}\right)  \right)^{2}.
   \end{aligned} $$
On the other hand, by the conclusion (c) in Lemma \ref{pghcz236131}, we have $$\left|\Xi_{j} \left( \textbf{z}\right)  -\Xi_{j} \left( \textbf{w}\right)  \right|  \leqslant C_{j,k}\left| \left| \textbf{z}-\textbf{w}\right|  \right|  ,\q\forall\; \textbf{z},\textbf{w}\in V^{o}_{k},$$where $C_{j,k}\triangleq \frac{1}{t_{j}} \cdot \left( 2\sqrt{t_{j}\sup\limits_{\ell^{2} } \left| I_{j}\eta \right|  } +\sup\limits_{\xi \in V^{o}_{k}} \left| \left| \xi \right|  \right|  \right) $.
 Hence, by (\ref{240122e7})	once more,
	$$
	\begin{aligned}
		&\frac{1}{2\pi } \int^{2\pi }_{0} \left( \psi \left(\Xi_{j} \left( \textbf{z}\right)  \right)  -\psi \left(\Xi_{j} \left( \textbf{z}+re^{\sqrt{-1}\theta }\zeta\right)  \right)  \right)  d\theta
\\&\geqslant \psi^{\prime } \left(\Xi_{j} \left( \textbf{z}\right)  \right) \frac{1}{2\pi } \int^{2\pi }_{0} \left(\Xi_{j} \left( \textbf{z}\right)  -\Xi_{j} \left( \textbf{z}+re^{\sqrt{-1}\theta }\zeta\right)  \right)  d\theta  -\frac{c^2_j}{4\pi } \int^{2\pi }_{0} \left(\Xi_{j} \left( \textbf{z}\right)  -\Xi_{j} \left( \textbf{z}+re^{\sqrt{-1}\theta }\zeta \right)  \right)^{2}  d\theta
 \\&\geqslant -c_j^1 \cdot C^{\prime }_{j}r^{2}\left| \left| \zeta \right|  \right|^{2}  -\frac{c^2_j}{4\pi } \int^{2\pi }_{0} \left|C_{j,k}\right|^2\left| \left| re^{\sqrt{-1}\theta }\zeta \right|  \right|^{2}  d\theta=-\left( c_j^1 \cdot C^{\prime }_{j}+\frac{c_j^2 }{2}\left|C_{j,k}\right|^2 \right)  r^{2}\left| \left| \zeta \right|  \right|^{2}.
	\end{aligned}
	$$
By Definition \ref{bthshdy23611}, this means that $\psi \left({\cal U}_{t_{j}}\left( I_{j}\eta \right)  +3-j\right)$ is globally semi-anti-plurisubharmonic on $V^{o}_{k}$ for all $k\in \mathbb{N}$.
	
	By $\Psi \left( \cdot\right)  \geqslant \eta \left( \cdot\right)$, it is clear that $\Psi(\cdot)>0$ and $\left\{ \textbf{z}\in V:\Psi \left( \textbf{z}\right)  \leqslant t\right\}  \stackrel{\circ}{\subset} V$ for all $t\in \mathbb{R}$.	For any $S\stackrel{\circ}{\subset} V$, there exists $k\in \mathbb{N}$ so that $S\stackrel{\circ}{\subset} V_{k}^{o}$ by Proposition \ref{jhbh23619}. Hence, $\Psi \left( \textbf{z}\right)  =\Psi_{k+2} \left( \textbf{z}\right)  $ for all $\textbf{z}\in V^{o}_{k}$, which means that $\Psi$ is a semi-anti-plurisubharmonic function on $V$, satisfying the last inequality in (\ref{240121e6}).
	This completes the proof of Theorem \ref{shtqj23612}.
\end{proof}

\section{Characterizing pseudo-convex domains by smooth plurisubharmonic functions}\label{section5}

First of all, by \cite[the equivalent condition (c) in Theorem 37.5, p. 274]{Mujica}, we have the following  characterization of pseudo-convex domains by Lipschitz plurisubharmonic functions.
\begin{proposition}\label{jntyddy236111}
Assume $V\not=\ell^2$. Then $V$ is a pseudo-convex domain  in $\ell^{2}$ if and only if the function $-\ln d\left( \cdot,\partial V\right)  $ is plurisubharmonic on $V$.
\end{proposition}




Next, we have the following simple result.
\begin{lemma}\label{xntjnt236145}
If there exists a plurisubharmonic exhaustive function $\eta\in C^2\left( V\right)  $ on $V$, then $V$ is a pseudo-convex domain.
\end{lemma}
\begin{proof}
For each finite dimensional subspace $M\subset \ell^{2}$, by the conclusion $(1)$ of  Proposition  \ref{dchcthhshzyxwdxzh23616}, $\eta |_{M\cap V}\in C\left( M\cap V\right)  $ is a plurisubharmonic function on $M\cap V$. For each $t\in \mathbb{R}$, the set $V_{t}$ (defined in (\ref{gx2401132})) is uniformly included in $V$. Hence, the set
 $\left\{ \textbf{z}\in M\cap V:\eta |_{M\bigcap V}\left( \textbf{z}\right)  < t\right\}  =M\cap V_{t}^0 $ is relatively compact in $M\bigcap V$, and therefore $M\cap V$ is pseudo-convex in $M$. By Definition \ref{xntyddy23618}, $V$ is a pseudo-convex domain of $\ell^{2}$.	
	This completed the proof of Theorem \ref{xntjnt236145}.
\end{proof}

Conversely, we have the following result, which can be regarded as an infinite-dimensional version of \cite[Theorem 2.6.11, p. 48]{Hor90}.
\begin{theorem}\label{jntyshxnty23614}
If $V$ is a pseudo-convex domain of $\ell^{2}$, then there exists an exhaustion function $\eta  \in C^{\infty }_{F^{\infty }}\left( V\right) (\subset C^2\left( V\right) ) $ such that $\eta  $ is plurisubharmonic on $V$ .
\end{theorem}
\begin{proof}
Without loss of generality, we consider only the case that $V\neq \ell^{2}$. We borrow some idea from the proof of \cite[Theorem 2.6.11, pp. 48--49]{Hor90}. Our proof is lengthy and therefore we divide it into several steps.

\medskip

	\textbf{Step 1}. 	Let
$$\varrho \left( \textbf{z}\right)  \triangleq \begin{cases}\left| \left| \textbf{z}\right|  \right|^{2}-\ln d\left( \textbf{z},\partial V\right)+c_0  ,&\textbf{z}\in V\\ 0,&\textbf{z}\notin V\end{cases}   ,$$
where $c_0$ is a positive number large enough so that $\varrho \left( \textbf{z}\right)  >0$ for all $\textbf{z}\in V$.
By the proof of Theorem \ref{shtqj23612}, there is a semi-anti-plurisubharmonic function $\Psi\geqslant \varrho$ on $V$, which satisfies (\ref{240121e6}).

For each $t>0$, recall the first paragraph in Section \ref{sec3} for the notations $V_{\Psi,t},V^{o}_{\Psi,t},  V_{\varrho ,t}$ and $V^{o}_{\varrho ,t}$.
By Proposition \ref{jhbh23619} and Remark \ref{jhbh236110}, it is clear that $V^{o}_{\Psi,t}\stackrel{\circ}{\subset} V^{o}_{\Psi,s}$ and $V^{o}_{\varrho,t}\stackrel{\circ}{\subset} V^{o}_{\varrho,s}$ for all $s>t>0$. Define a strictly increasing function $\lambda :\left( 0,+\infty \right)  \rightarrow \left( 0,+\infty
	\right)  $ by
$$
\lambda(t)\triangleq 1+\inf\left\{s\geqslant t:\;V^{o}_{\varrho,t}\subset V^{o}_{\Psi,s}\right\}.
$$
It is obvious that $\lambda \left( t\right) > t$ and $V_{\Psi,t}^{o}\subset V_{\varrho,t}^{o}\subset V_{\Psi,\lambda \left( t\right) }^{o}$ for all $t>0$.
	
For any $t>0$, there exists $C(t)>0$ so that for any $ \textbf{z},\textbf{w}\in V^{o}_{\varrho,t}$, it holds that
\begin{equation}\label{bdsh236146}
		\left\{\begin{aligned}
			&\left| \varrho \left( \textbf{z}\right)  -\varrho \left( \textbf{w}\right)  \right|  \leqslant C\left( t\right)  \left| \left| \textbf{z}-\textbf{w}\right|  \right|  ,\q\left| \varrho \left( \textbf{z}\right)  \right|  \leqslant C\left( t\right) ,\\&\left| \Psi \left( \textbf{z}\right)  -\Psi \left( \textbf{w}\right)  \right|  \leqslant C\left( t\right)  \left| \left| \textbf{z}-\textbf{w}\right|  \right|  ,\q\left| \Psi \left( \textbf{z}\right)  \right|  \leqslant C\left( t\right),
		\end{aligned}\right.
\end{equation}
 and for each $\textbf{z}\in V_{\Psi,t}^{o}, \zeta\in \ell^{2}$ and $\delta>0$, there is $r\in (0,\delta)$ so that $\textbf{z}+\rho e^{\sqrt{-1}\theta }\zeta \in V^{o}_{\Psi,t}$ for all $(\rho,\theta ) \in \left[ 0,r\right]  \times \left[ 0,2\pi \right]  $, and
 \begin{equation}\label{lsh236138}
		\frac{1}{2\pi } \int^{2\pi }_{0}\left(\Psi \left( \textbf{z}\right)  -\Psi \left( \textbf{z}+re^{\sqrt{-1}\theta }\zeta \right)\right)d\theta \geqslant -C\left( t\right)  \cdot r^{2}\left| \left| \zeta\right|  \right|^{2} .
	\end{equation}
	
	For each $j\in \mathbb{N}$, we choose $\varepsilon_{j} \in \left( 0,\inf\limits_{\textbf{z}\in V^{o}_{\Psi,\lambda(j)}} \frac{1}{1+\left| \left| \textbf{z}\right|  \right|^{2}   +C\left( \lambda \left( j\right)  +{1}/{2} \right)  } \right)  $ so that
	
	\begin{equation}\label{fw1236147}
		\textbf{z}-\varepsilon_{j} \zeta \in V^{o}_{\Psi,\lambda \left( j\right)  +{1}/{2} },\q  \forall\; \left(\textbf{z},\zeta\right)\in V^{o}_{\Psi,\lambda \left( j\right)  }\times \overline{B_1},
			\end{equation}
	\begin{equation}
		\quad \quad \quad  \textbf{z}-\varepsilon_{j} \zeta \notin V^{o}_{\Psi,\lambda \left( j\right)  +2},\q  \forall\; \left(\textbf{z},\zeta\right)\in \left(\ell^{2}\setminus V^{o}_{\Psi,\lambda \left( j\right)  +3}\right)\times \overline{B_1}, \label{fw1236148}
	\end{equation}
	\begin{equation}\label{fw1236149}
		\quad \textbf{z}-\varepsilon_{j} \zeta \in V^{o}_{\Psi,\lambda \left( j\right)  +5},\q  \forall\;\left( \textbf{z},\zeta\right)\in V^{o}_{\Psi,\lambda \left( j\right)  +4}\times\overline{B_1}.
	\end{equation}
	We set  (Recall (\ref{lsh23611912}) and (\ref{240122e9}) respectively for $\vartheta \left( \cdot\right)$ and $c$)
$$\varrho_{j} \left( \textbf{z}\right)  \triangleq c\int_{\ell^{2} } I_{\lambda \left( j\right)  +1}\left( \textbf{z}-\varepsilon_{j} \zeta \right)  \varrho \left( \textbf{z}-\varepsilon_{j} \zeta \right)  \vartheta \left( \zeta \right)  dP\left( \zeta \right) +\varepsilon_{j} \left| \left| \textbf{z}\right|  \right|^{2}   ,\q\forall\;\textbf{z}\in \ell^{2},
 $$
where $I_{\lambda \left( j\right)  +1}\left( \cdot\right)$ is given by (\ref{lsh23611912e1}) with $\tau$ and $\eta$ therein replaced by $\lambda \left( j\right)  +1$ and $\Psi$, respectively.
By \eqref{fw1236148}, we have $I_{\lambda \left( j\right)  +1}\left( \textbf{z}-\varepsilon_{j} \zeta \right)  \varrho \left( \textbf{z}-\varepsilon_{j} \zeta \right)  \vartheta \left( \zeta \right)  =0$ for all $\left(\textbf{z},\zeta \right)\in \left(\ell^{2}\setminus V^{o}_{\Psi,\lambda \left( j\right)  +3}\right)\times \ell^{2}$ , which means that
 \begin{equation}\label{fw1236149e1}
  \supp \left( \varrho_{j} (\cdot)-\varepsilon_{j} \left| \left|\cdot\right|  \right|^{2}\right)  \subset V_{\Psi,\lambda \left( j\right)  +3}.
   \end{equation}
 Similar to the proof of Theorem \ref{ghqjhshczx23613}, we can show that
  $$\sup_{\textbf{z},\textbf{w}\in \ell^{2} ,\textbf{z}\neq \textbf{w}} \frac{\left| I_{\lambda \left( j\right)  +1}\left( \textbf{z}\right)  \varrho \left( \textbf{z}\right)  -I_{\lambda \left( j\right)  +1}\left( \textbf{w}\right)  \varrho \left( \textbf{w}\right)  \right|  }{\left| \left| \textbf{z}-\textbf{w}\right|  \right|  } <\infty .$$Therefore, by Theorem \ref{pmgjsh236118} and $\varepsilon_{j} \left| \left|\cdot\right|  \right|^{2}  \in C^{\infty }_{F^{\infty}}\left( V\right)  $, we conclude that  $\varrho_{j} \in C^{\infty }_{F^{\infty}}\left( V\right)$.
	
\medskip

	\textbf{Step 2}.	
	In this step, we will prove that for each fixed $k\in\mathbb{N}$ and $ j \in \mathbb{N}$, it holds that \begin{eqnarray}
		\inf_{n\in \mathbb{N} ,\textbf{z}\in V^{o}_{\Psi,\lambda \left( k\right)  },\sum\limits^{n}_{i=1} \left| w_{i}\right|^{2}  =1,w_{1},w_{2},\cdots ,w_{n}\in \mathbb{C}} \sum_{i,\ell=1}^n  \partial_{i} \overline{\partial_{\ell } }  \varrho_{j} \left( \textbf{z}\right)  w_{i}\overline{w_{\ell }} >-\infty.\label{ejdsh236150}
	\end{eqnarray}
	
	By applying Lemma \ref{bthsh236136} to $-\varrho_{j}$, we see that, it suffices to find $C_{j,k}>0$ so that for each $\textbf{z}\in V^{o}_{\Psi,\lambda \left( k\right)  }, \zeta \in \ell^{2}$ and $\delta \in \left( 0,1\right)  $, there exists  $r\in (0,\delta)$ such that $\textbf{z}+\rho e^{\sqrt{-1}\theta } \zeta \in V^{o}_{\Psi,\lambda \left( k\right)  }$ for all $\rho \in \left[ 0,r\right] $ and $\theta \in \left[ 0,2\pi \right]  $, and that
 \begin{equation}\label{fw1236149e1e}\frac{1}{2\pi } \int^{2\pi }_{0} \left( \varrho_{j} \left( \textbf{z}\right)  -\varrho_{j} \left( \textbf{z}+re^{\sqrt{-1}\theta }\zeta \right)  \right) d\theta \leqslant C_{j,k}r^{2}  \left| \left| \zeta \right|  \right|^{2}.
 \end{equation}
	
	Noting (\ref{fw1236149e1}), we find that
  $$
  \begin{aligned}
			&\frac{1}{2\pi } \int^{2\pi }_{0} \left( \varrho_{j} \left( \textbf{z}\right)  -\varrho_{j} \left( \textbf{z}+re^{\sqrt{-1}\theta }\zeta \right)  \right)  d\theta=\frac{\varepsilon_{j}}{2\pi } \int^{2\pi }_{0} \left( \left| \left| \textbf{z}\right|  \right|^{2}  -\left| \left| \textbf{z}+re^{\sqrt{-1}\theta }\zeta \right|  \right|^{2}  \right)d\theta \\&=-\frac{\varepsilon_{j}}{2\pi } \int^{2\pi }_{0} \left(\left| r\right|^{2}  \left| \left| \zeta \right|  \right|^{2}  +2\Re\left( \textbf{z},re^{\sqrt{-1}\theta }\zeta \right)\right)d\theta=-\varepsilon_{j} \left| r\right|^{2}  \left| \left| \zeta \right|  \right|^{2}   \leqslant 0
		\end{aligned}
	$$
holds for each $\textbf{z}\in V^{o}_{\Psi,\lambda \left( k\right)  }\setminus V_{\Psi,\lambda \left( j\right)  +3},\zeta \in \ell^{2} $ and sufficiently small $r>0$ so that $\textbf{z}+\rho e^{\sqrt{-1}\theta }\zeta \in V^{o}_{\Psi,\lambda \left( k\right)  }\setminus V_{\Psi,\lambda \left( j\right)  +3}$ for all $\rho \in \left[ 0,r\right]$ and $\theta \in \left[ 0,2\pi \right]  $.
	
For each $\textbf{z}\in V^{o}_{\Psi,\lambda \left( k\right)  }\bigcap V_{\Psi,\lambda \left( j\right)  +3}, \textbf{w}\in \ell^{2}$ and $\delta>0$, by \eqref{fw1236149}, we have $\textbf{z}-\varepsilon_{j} \zeta\in V^{o}_{\Psi,\lambda \left( j\right)  +5}$ for all $\zeta\in \overline{B_1}$. Then  by $V^{o}_{\Psi,\lambda \left( j\right)  +5} \stackrel{\circ}{\subset} V^{o}_{\Psi,\lambda \left( j\right)  +6}$ and the inequality \eqref{lsh236138}, we may choose $r\in (0,\delta)$ so that $\textbf{z}-\varepsilon_{j} \zeta +\rho e^{\sqrt{-1}\theta }\textbf{w}\in V^{o}_{\Psi,\lambda \left( j\right)  +6}$ and $\textbf{z} +\rho e^{\sqrt{-1}\theta }\textbf{w}\in V^{o}_{\Psi,\lambda \left( k\right)  }$ for all $\rho \in \left[ 0,r\right]  ,\theta \in \left[ 0,2\pi \right] $ and $\zeta\in \overline{B_1}$, and
 $$\frac{1}{2\pi } \int^{2\pi }_{0} \left( \Psi \left( \textbf{z}+re^{\sqrt{-1}\theta }\textbf{w}-\varepsilon_{j} \zeta \right)  -\Psi \left( \textbf{z}-\varepsilon_{j} \zeta \right)  \right)  d\theta \leqslant C\left( \lambda \left( j\right)  +6\right)  r^{2}\left| \left| \textbf{w}\right|  \right|^{2}  ,\q\forall \;\zeta\in \overline{B_1}.$$
Since $V$ is  pseudo-convex, by Proposition \ref{jntyddy236111}, $-\ln d\left( \cdot,\partial V\right) $ is plurisubharmonic on $V$, hence so is $\varrho$ on $V$. By Proposition \ref{dchcthhshjbd23617}, we have
 $$
		\begin{aligned}
			&\frac{c}{2\pi } \int^{2\pi }_{0} \int_{\ell^{2} } I_{\lambda \left( j\right)  +1}\left( \textbf{z}-\varepsilon_{j} \zeta \right) \left(\varrho \left( \textbf{z}+re^{\sqrt{-1}\theta }\textbf{w}-\varepsilon_{j} \zeta \right)  -\varrho \left( \textbf{z}-\varepsilon_{j} \zeta \right)  \right)\vartheta \left( \zeta \right)  dP\left( \zeta \right)  d\theta \\&=c\int_{\ell^{2} } I_{\lambda \left( j\right)  +1}\left( \textbf{z}-\varepsilon_{j} \zeta \right)  \cdot\frac{1}{2\pi } \int^{2\pi }_{0} \left(\varrho \left( \textbf{z}+re^{\sqrt{-1}\theta }\textbf{w}-\varepsilon_{j} \zeta \right)  -\varrho \left( \textbf{z}-\varepsilon_{j} \zeta \right)  \right)d\theta \cdot\vartheta \left( \zeta \right)  dP\left( \zeta \right)  \geqslant 0.
		\end{aligned}
	$$
	Therefore, by the Taylor mean value theorem, \eqref{bdsh236146} and $ -K_0\leqslant \mathcal{I}^{\prime }_{\lambda \left( j\right)  +1} \leqslant 0$ (from \eqref{shyhshdxzh236114}),
	we have
 $$
		\begin{aligned}
			&\frac{c}{2\pi } \int^{2\pi }_{0} \int_{\ell^{2} } \left( (I_{\lambda \left( j\right)  +1}\varrho )\left( \textbf{z}+re^{\sqrt{-1}\theta }\textbf{w}-\varepsilon_{j} \zeta \right)  -(I_{\lambda \left( j\right)  +1}\varrho )\left( \textbf{z}-\varepsilon_{j} \zeta \right)  \right) \vartheta \left( \zeta \right)  dP\left( \zeta \right)  d\theta  \\&  =\frac{c}{2\pi } \int^{2\pi }_{0} \int_{\ell^{2} } \left( I_{\lambda \left( j\right)  +1}\left( \textbf{z}+re^{\sqrt{-1}\theta }\textbf{w}-\varepsilon_{j} \zeta \right)  -I_{\lambda \left( j\right)  +1}\left( \textbf{z}-\varepsilon_{j} \zeta \right)  \right) \\&\qq\qq\qq\times\varrho \left( \textbf{z}+re^{\sqrt{-1}\theta }\textbf{w}-\varepsilon_{j} \zeta \right)  \vartheta \left( \zeta \right)  dP\left( \zeta \right)  d\theta \\& \quad+\frac{c}{2\pi } \int^{2\pi }_{0} \int_{\ell^{2} } I_{\lambda \left( j\right)  +1}\left( \textbf{z}-\varepsilon_{j} \zeta \right)\left(\varrho \left( \textbf{z}+re^{\sqrt{-1}\theta }\textbf{w}-\varepsilon_{j} \zeta \right)  -\varrho \left( \textbf{z}-\varepsilon_{j} \zeta \right)\right)\vartheta \left( \zeta \right)  dP\left( \zeta \right)  d\theta \\&\geqslant \frac{c}{2\pi } \int^{2\pi }_{0} \int_{\ell^{2} } \left( I_{\lambda \left( j\right)  +1}\left( \textbf{z}+re^{\sqrt{-1}\theta }\textbf{w}-\varepsilon_{j} \zeta \right)  -I_{\lambda \left( j\right)  +1}\left( \textbf{z}-\varepsilon_{j} \zeta \right)  \right) \\&\qq\qq\qq\times \varrho \left( \textbf{z}+re^{\sqrt{-1}\theta }\textbf{w}-\varepsilon_{j} \zeta \right)  \vartheta \left( \zeta \right)  dP\left( \zeta \right)  d\theta\\&= \frac{c}{2\pi } \int^{2\pi }_{0} \int_{\ell^{2} } \left( I_{\lambda \left( j\right)  +1}\left( \textbf{z}+re^{\sqrt{-1}\theta }\textbf{w}-\varepsilon_{j} \zeta \right)  -I_{\lambda \left( j\right)  +1}\left( \textbf{z}-\varepsilon_{j} \zeta \right)  \right)  \varrho \left( \textbf{z}-\varepsilon_{j} \zeta \right)  \vartheta \left( \zeta \right)  dP\left( \zeta \right)  d\theta \\&\quad+ \frac{c}{2\pi } \int^{2\pi }_{0} \int_{\ell^{2} } \left( I_{\lambda \left( j\right)  +1}\left( \textbf{z}+re^{\sqrt{-1}\theta }\textbf{w}-\varepsilon_{j} \zeta \right)  -I_{\lambda \left( j\right)  +1}\left( \textbf{z}-\varepsilon_{j} \zeta \right)  \right)\\
&\qq\qq\qq\q\cdot\left(\varrho \left( \textbf{z}+re^{\sqrt{-1}\theta }\textbf{w}-\varepsilon_{j} \zeta \right)  -\varrho \left( \textbf{z}-\varepsilon_{j} \zeta\right)\right)\vartheta \left( \zeta \right)  dP\left( \zeta \right)  d\theta\\
			&\geqslant \frac{c}{2\pi } \int_{\ell^{2} } \mathcal{I}^{\prime }_{\lambda \left( j\right)  +1} \left( \Psi \left( \textbf{z}-\varepsilon_{j} \zeta \right)  \right)  \\&\qq\qq\qq\times\int^{2\pi }_{0} \left( \Psi \left( \textbf{z}+re^{\sqrt{-1}\theta }\textbf{w}-\varepsilon_{j} \zeta \right)  -\Psi \left( \textbf{z}-\varepsilon_{j} \zeta \right)  \right)  d\theta\varrho \left( \textbf{z}-\varepsilon_{j} \zeta \right)  \vartheta \left( \zeta \right)  dP\left( \zeta \right)   \\&\quad-\frac{c\sup\limits_{\mathbb{R} } \left| \mathcal{I}^{\prime \prime }_{\lambda \left( j\right)  +1} \right|  }{4\pi } \int^{2\pi }_{0} \int_{\ell^{2} } \left( \Psi \left( \textbf{z}+re^{\sqrt{-1}\theta }\textbf{w}-\varepsilon_{j} \zeta \right)  -\Psi \left( \textbf{z}-\varepsilon_{j} \zeta \right)  \right)^{2}  \varrho \left( \textbf{z}-\varepsilon_{j} \zeta \right)  \vartheta \left( \zeta \right)  dP\left( \zeta \right)  d\theta  \\&\quad-\frac{c}{2\pi } \int^{2\pi }_{0} \int_{\ell^{2} } \left| I_{\lambda \left( j\right)  +1}\left( \textbf{z}+re^{\sqrt{-1}\theta }\textbf{w}-\varepsilon_{j} \zeta \right)  -I_{\lambda \left( j\right)  +1}\left( \textbf{z}-\varepsilon_{j} \zeta \right)  \right| \\
&\qq\qq\qq\q \cdot \left| \varrho \left( \textbf{z}+re^{\sqrt{-1}\theta }\textbf{w}-\varepsilon_{j} \zeta \right)  -\varrho \left( \textbf{z}-\varepsilon_{j} \zeta\right)  \right|  \vartheta \left( \zeta \right)  dP\left( \zeta \right)  d\theta
	\\&\geqslant -cK_0C\left( \lambda \left( j\right)  +6\right)  \sup_{V^{o}_{\Psi,\lambda \left( j\right)  +6}} \left| \varrho \right|  \int_{\ell^{2} }r^{2}  \left| \left| \textbf{w}\right|  \right|^{2}  \vartheta \left( \zeta \right)  dP\left( \zeta \right)
 \\& \quad-\frac{c\sup\limits_{\mathbb{R} } \left| \mathcal{I}^{\prime \prime }_{\lambda \left( j\right)  +1} \right|  \sup\limits_{V^{o}_{\Psi,\lambda \left( j\right)  +6}} \left| \varrho \right|  \left|C\left( \lambda \left( j\right)  +6\right) \right|^{2} }{4\pi } \int^{2\pi }_{0} \int_{\ell^{2} } r^{2}  \left| \left| \textbf{w}\right|  \right|^{2}  \vartheta \left( \zeta \right)  dP\left( \zeta \right)  d\theta
  \\&\quad-\frac{cK_0\left|C\left( \lambda \left( j\right)  +6\right) \right|^{2} }{2\pi } \int^{2\pi }_{0} \int_{\ell^{2} }r^{2}  \left| \left| \textbf{w}\right|  \right|^{2}  \vartheta \left( \zeta \right)  dP\left( \zeta \right)  d\theta
 \\&=\Big( -K_0C\left( \lambda \left( j\right)  +6\right)  \sup_{V^{o}_{\Psi,\lambda \left( j\right)  +6}} \left| \varrho \right|  -\frac{1}{2} \sup\limits_{\mathbb{R} } \left| \mathcal{I}^{\prime \prime }_{\lambda \left( j\right)  +1} \right|  \sup\limits_{V^{o}_{\Psi,\lambda \left( j\right)  +6}} \left| \varrho \right| \left|C\left( \lambda \left( j\right)  +6\right) \right|^{2} \\&\qq -K_0\left|C\left( \lambda \left( j\right)  +6\right)\right|^{2} \Big) r^{2}  \left| \left| \textbf{w}\right|  \right|^{2} .
	\end{aligned}$$
	Therefore, the desired inequality (\ref{fw1236149e1e}) holds,
	which means that \eqref{ejdsh236150} holds  for all $k\in \mathbb{N}$ and $j \in \mathbb{N}$.
	
	By Lemma \ref{bthsh236136}, more specifically, for any $k\in \mathbb{N}$ and $j \in \mathbb{N}$, we have
 \begin{eqnarray}\label{lsh236139}
		\inf_{n\in \mathbb{N} ,\textbf{z}\in V^{o}_{\Psi,\lambda \left( k\right)  },\sum\limits^{n}_{i=1} \left| w_{i}\right|^{2}  =1,w_{1},w_{2},\cdots ,w_{n}\in \mathbb{C} } \sum_{i,\ell=1}^n  \partial_{i} \overline{\partial_{\ell } } \varrho_{j} \left( \textbf{z}\right)  w_{i}\overline{w_{\ell }} \geqslant -C_{j,k}.
	\end{eqnarray}

\medskip
	
	\textbf{Step 3}.
	In this step, we will prove that $\varrho_{j} (\cdot)-\varepsilon_{j} \left| \left| \cdot\right|  \right|^{2}  $ is plurisubharmonic on $V^{o}_{\Psi,\lambda \left( j\right)  }$ for each $j\in \mathbb{N}$ and
 \begin{eqnarray}\label{qjhsh236151}
		\varrho \left( \textbf{z}\right)  \leqslant \varrho_{j} \left( \textbf{z}\right)  \leqslant \varrho \left( \textbf{z}\right)  +1,\q\forall\; \textbf{z}\in V^{o}_{\Psi,\lambda \left( j\right)  }.
	\end{eqnarray}
	
	For each $j\in \mathbb{N}$, $a\in V^{o}_{\Psi,\lambda \left( j\right)  }$ and $ b\in \ell^{2}$, noting that $\varrho$  is plurisubharmonic  on $V$,  by Proposition \ref{dchcthhshjbd23617} and  \eqref{fw1236147}, we have $I_{\lambda \left( j\right)  +1}\left( a-\varepsilon_{j} \zeta \right)  \varrho \left( a-\varepsilon_{j} \zeta \right)  \vartheta \left( \zeta \right)  =\varrho \left( a-\varepsilon_{j} \zeta \right)  \vartheta \left( \zeta \right) $, and that
	$$
		\begin{aligned}
			&c\int_{\ell^{2} } I_{\lambda \left( j\right)  +1}\left( a-\varepsilon_{j} \zeta \right)  \varrho \left( a-\varepsilon_{j} \zeta \right)  \vartheta \left( \zeta \right)  dP\left( \zeta \right)  =c\int_{\ell^{2} } \varrho \left( a-\varepsilon_{j} \zeta \right)  \vartheta \left( \zeta \right)  dP\left( \zeta \right)   \\&\leqslant \frac{c}{2\pi } \int_{\ell^{2} } \int^{2\pi }_{0} \varrho \left( a+re^{\sqrt{-1}\theta }b-\varepsilon_{j} \zeta \right)  d\theta \cdot \vartheta \left( \zeta \right)  dP\left( \zeta \right)   \\&=\frac{c}{2\pi } \int^{2\pi }_{0} \int_{\ell^{2} } \varrho \left( a+re^{\sqrt{-1}\theta }b-\varepsilon_{j} \zeta \right)  \vartheta \left( \zeta \right)  dP\left( \zeta \right)  d\theta  \\&=\frac{c}{2\pi } \int^{2\pi }_{0} \int_{\ell^{2} } I_{\lambda \left( j\right)  +1}\left( a+re^{\sqrt{-1}\theta }b-\varepsilon_{j} \zeta \right)  \varrho \left( a+re^{\sqrt{-1}\theta }b-\varepsilon_{j} \zeta \right)  \vartheta \left( \zeta \right)  dP\left( \zeta \right)  d\theta,
		\end{aligned}
	$$
	where $r>0$ is sufficiently small so that $a+\rho e^{\sqrt{-1}\theta }b-\varepsilon_{j} \zeta \in V^{o}_{\Psi,\lambda \left( j\right) +1 }$ for all $\zeta \in \overline{B_1}$, $\rho \in \left[ 0,r\right]$ and $\theta \in \left[ 0,2\pi \right]  $. Hence,  $\varrho_{j} (\cdot)-\varepsilon_{j} \left| \left| \cdot\right|  \right|^{2}  $ is plurisubharmonic  on $V^{o}_{\Psi,\lambda \left( j\right)  }$.

	For each $a\in V^{o}_{\Psi,\lambda \left( j\right)  }$, by \eqref{fw1236147}, we have $I_{\lambda \left( j\right)  +1}\left( a-\varepsilon_{j} \zeta \right)  =1$ for all $\zeta \in \overline{B_1}$. Hence
	$$
		\begin{aligned}
			\varrho_{j} \left( a\right)  &=\frac{c}{2\pi } \int^{2\pi }_{0} \int_{\overline{B_1}} I_{\lambda \left( j\right)  +1}\left( a-\varepsilon_{j} \zeta \right)  \varrho \left( a-\varepsilon_{j} \zeta \right)  \vartheta \left( \zeta \right)  dP\left( \zeta \right)  d\theta +\varepsilon_{j} \left| \left| a\right|  \right|^{2}  \\& \geqslant\frac{c}{2\pi } \int^{2\pi }_{0} \int_{\overline{B_1}} I_{\lambda \left( j\right)  +1}\left( a-\varepsilon_{j} \zeta \right)  \varrho \left( a-\varepsilon_{j} \zeta \right)  \vartheta \left( \zeta \right)  dP\left( \zeta \right)  d\theta \\&=\frac{c}{2\pi } \int^{2\pi }_{0} \int_{\overline{B_1}} \varrho \left( a-\varepsilon_{j} \zeta \right)  \vartheta \left( \zeta \right)  dP\left( \zeta \right)  d\theta =\frac{c}{2\pi } \int_{\overline{B_1}} \int^{2\pi }_{0} \varrho \left( a-\varepsilon_{j} e^{\sqrt{-1}\theta }\zeta \right)  d\theta \cdot \vartheta \left( \zeta \right)  dP\left( \zeta \right)   ,
		\end{aligned}
	$$
	where the last equality follows from Lemma \ref{jbysdsh236115}. By $\varrho \left( a+\lambda \left( -e^{\sqrt{-1}\theta }\zeta \right)  \right)  $ is subharmonic  with respect to $\lambda\in \mathbb{C}$ and using \cite[Theorem 17.5, p. 337]{Rud87} , we see that
	$\frac{1}{2\pi } \int^{2\pi }_{0} \varrho \left( a+\lambda \left( -e^{\sqrt{-1}\theta }\zeta \right)  \right)  d\theta $ is increasing with respect to $\lambda \in [ 0,\varepsilon_{j} ]  $. Hence
 $$
		\begin{aligned}
			&\frac{c}{2\pi } \int_{\overline{B_1}} \int^{2\pi }_{0} \varrho \left( a-\varepsilon_{j} e^{\sqrt{-1}\theta }\zeta \right)  d\theta\cdot\vartheta \left( \zeta \right)   dP\left( \zeta \right) \geqslant \frac{c}{2\pi } \int_{\overline{B_1}} \lim_{\varepsilon \rightarrow 0^{+}} \int^{2\pi }_{0} \varrho \left( a-\varepsilon e^{\sqrt{-1}\theta }\zeta \right)  d\theta \cdot \vartheta \left( \zeta \right)  dP\left( \zeta \right) \\&\geqslant \frac{c}{2\pi } \int_{\overline{B_1}} \int^{2\pi }_{0} \varliminf_{\varepsilon \rightarrow 0^{+}} \varrho \left( a-\varepsilon e^{\sqrt{-1}\theta }\zeta \right)  d\theta \cdot \vartheta \left( \zeta \right)  dP\left( \zeta \right)= c\varrho \left( a\right)  \int_{\overline{B_1}} \vartheta \left( \zeta \right)  dP\left( \zeta \right)  =\varrho (a),
		\end{aligned}
	$$
	where the last inequality follows from the Fatou lemma. Hence $\varrho_{j} \left( a\right) \geqslant\varrho (a)$.
	
	By \eqref{fw1236147}, for each $ \zeta\in \overline{B_1}$ and $\textbf{z}\in V^{o}_{\Psi,\lambda \left( j\right)  }$, it follows that
    $$
		\varrho \left( \textbf{z}-\varepsilon_{j} \zeta \right)  \leqslant \left| \varrho \left( \textbf{z}-\varepsilon_{j} \zeta \right)  -\varrho \left( \textbf{z}\right)  \right|  +\varrho \left( \textbf{z}\right)  \leqslant C\left(\lambda \left( j\right)  +{1}/{2} \right)  \varepsilon_{j} +\varrho \left( \textbf{z}\right)   .
	$$
	Combining the above with $\varepsilon_{j} \in \left( 0,\inf\limits_{\textbf{z}\in V^{o}_{\Psi,\lambda(j)}} \frac{1}{1+\left| \left| \textbf{z}\right|  \right|^{2} +C\left( \lambda \left( j\right)  +{1}/{2} \right)  } \right)  $, we have
	$$
\varrho_{j} \left( \textbf{z}\right)  \leqslant c\int_{\overline{B_1}} \left( C\left( \lambda \left( j\right)  +{1}/{2} \right)  \varepsilon_{j} +\varrho \left( \textbf{z}\right)  \right)  \vartheta \left( \zeta \right)  dP\left( \zeta \right)  +\varepsilon_{j} \left| \left| \textbf{z}\right|  \right|^{2}\leqslant 1+\varrho \left( \textbf{z}\right)  ,\q\forall\; \textbf{z}\in V^{o}_{\Psi,\lambda \left( j\right)  }.$$
Hence, (\ref{qjhsh236151}) holds for each $j\in \mathbb{N}$.
	
	\medskip

\textbf{Step 4}.
	We now construct the required plurisubharmonic exhaustion function $\eta  \in C^{\infty }_{F^{\infty }}\left( V\right)  $.
	
	Let $\psi $ be a function satisfying \eqref{240122e1}. We choose below inductively positive numbers $\alpha_{1},\alpha_{2},\cdots $  and functions
 \begin{equation}\label{lsh236138e1}
 \eta _{k} \left( \cdot\right)  \triangleq \sum\limits^{k}_{j=1} \alpha_{j}\psi \left( \varrho_{j} \left(\cdot\right)  +2-j\right)
   \end{equation}
 such that $ \eta _{k}  \in C^{\infty }_{F^{\infty }}\left( V\right)    $  is plurisubharmonic on $V^{o}_{\Psi,\lambda \left( k\right)  }$ for each $k\in \mathbb{N}$, and
  \begin{equation}\label{lsh236138e2}
  \eta _{k} \left( \textbf{z}\right)   \geqslant \varrho \left( \textbf{z}\right)  ,\q\forall\; \textbf{z}\in V^{o}_{\Psi,\lambda \left( k\right)  }.
  \end{equation}

 Choose $\alpha_{1} \triangleq \frac{1}{\psi(1) }\sup\limits_{V_{\Psi,\lambda \left( 1\right)  }} \varrho $ and $\eta _{1} \left( \cdot\right)  \triangleq \alpha_{1}\psi \left( \varrho_{1} \left( \cdot\right)  +1\right) $. It is easy to check that $\eta _{1} \left( \textbf{z}\right)   \geqslant \varrho \left( \textbf{z}\right)$ for all $\textbf{z}\in V^{o}_{\Psi,\lambda \left( 1\right)  }$.
 Assume that some positive numbers $\alpha_{2},\cdots, \alpha_{k}$ are chosen so that the plurisubharmonic function $ \eta _{k} \left( \in C^{\infty }_{F^{\infty }}\left( V\right) \right)  $ on $V^{o}_{\Psi,\lambda \left( k\right)  }$ given by (\ref{lsh236138e1}) satisfies (\ref{lsh236138e2}).
Then, by \eqref{lsh236139},
$$
		\begin{aligned}
			s_k&\triangleq \inf_{n\in \mathbb{N} ,\textbf{z}\in V^{o}_{\Psi,\lambda \left( k+1\right)  }\setminus V^{o}_{\Psi,\lambda \left( k\right)  },\sum\limits^{n}_{i=1} \left| w_{i}\right|^{2}  =1,w_{1},\cdots ,w_{n}\in \mathbb{C}} \Big\{ \sum_{i,j=1}^n  \partial_{i} \overline{\partial_{j} } \eta _{k} \left( \textbf{z}\right)  w_{i}\overline{w_{j}} \Big\}
   \\&=\inf_{n\in \mathbb{N} ,\textbf{z}\in V^{o}_{\Psi,\lambda \left( k+1\right)  }\setminus V^{o}_{\Psi,\lambda \left( k\right)  },\sum\limits^{n}_{i=1} \left| w_{i}\right|^{2}  =1,w_{1},\cdots ,w_{n}\in \mathbb{C}} \Big\{ \sum^{k}_{\ell =1} \alpha_{\ell } \sum_{i,j=1}^n  \partial_{i} \overline{\partial_{j} } \psi \left( \varrho_{\ell } \left( \textbf{z}\right)  +2-\ell \right)  w_{i}\overline{w_{j}} \Big\}
    \\&\geqslant \sum^{k}_{\ell =1} \alpha_{\ell } \inf_{n\in \mathbb{N} ,\textbf{z}\in V^{o}_{\Psi,\lambda \left( k+1\right)  }\setminus V^{o}_{\Psi,\lambda \left( k\right)  },\sum\limits^{n}_{i=1} \left| w_{i}\right|^{2}  =1,w_{1},\cdots ,w_{n}\in \mathbb{C}} \Big\{ \sum_{i,j=1}^n  \partial_{i} \overline{\partial_{j} } \psi \left( \varrho_{\ell } \left( \textbf{z}\right)  +2-\ell \right)  w_{i}\overline{w_{j}} \Big\}
    \\&=\sum^{k}_{\ell =1} \alpha_{\ell } \inf_{n\in \mathbb{N} ,\textbf{z}\in V^{o}_{\Psi,\lambda \left( k+1\right)  }\setminus V^{o}_{\Psi,\lambda \left( k\right)  },\sum\limits^{n}_{i=1} \left| w_{i}\right|^{2}  =1, w_{1},\cdots ,w_{n}\in \mathbb{C}}\Big\{ \psi^{\prime \prime } \left( \varrho_{\ell } \left( \textbf{z}\right)  +2-\ell \right) \left|\sum_{1\leqslant j\leqslant n} \partial_{j} \varrho_{\ell } w_{j}\right|^{2}\Big\}  \\&\quad+ \sum^{k}_{\ell =1} \alpha_{\ell } \inf_{n\in \mathbb{N} ,\textbf{z}\in V^{o}_{\Psi,\lambda \left( k+1\right)  }\setminus V^{o}_{\Psi,\lambda \left( k\right)  },\sum\limits^{n}_{i=1} \left| w_{i}\right|^{2}  =1, w_{1},\cdots ,w_{n}\in \mathbb{C}}\Big\{ \psi^{\prime } \left( \varrho_{\ell } \left( \textbf{z}\right)  +2-\ell \right)  \sum_{i,j=1}^n  \partial_{i} \overline{\partial_{j} } \varrho_{\ell } \left( \textbf{z}\right)  \cdot w_{i}\overline{w_{j}} \Big\}
		\\&\geqslant \sum^{k}_{\ell =1} \alpha_{\ell } \inf_{n\in \mathbb{N} ,\textbf{z}\in V^{o}_{\Psi,\lambda \left( k+1\right)  }\setminus V^{o}_{\Psi,\lambda \left( k\right)  },\sum\limits^{n}_{i=1} \left| w_{i}\right|^{2}  =1, w_{1},\cdots ,w_{n}\in \mathbb{C}}\Big\{ \psi^{\prime } \left( \varrho_{\ell } \left( \textbf{z}\right)  +2-\ell \right)  \sum_{i,j=1}^n  \partial_{i} \overline{\partial_{j} } \varrho_{\ell } \left( \textbf{z}\right)  \cdot w_{i}\overline{w_{j}} \Big\}
\\&\geqslant \sum^{k}_{\ell =1} \alpha_{\ell } \inf_{\textbf{z}\in V^{o}_{\Psi,\lambda \left( k+1\right)  }\setminus V^{o}_{\Psi,\lambda \left( k\right)  }} \Big\{ \psi^{\prime } \left( \varrho_{\ell } \left( \textbf{z}\right)  +2-\ell \right)\\
  &\qq\qq\qq\qq\qq\qq\times \inf_{n\in \mathbb{N} ,\zeta \in V^{o}_{\Psi,\lambda \left( k+1\right)  },\sum\limits^{n}_{i=1} \left| w_{i}\right|^{2}  =1,w_{1},\cdots ,w_{n}\in \mathbb{C} } \sum_{i,j=1}^n  \partial_{i} \overline{\partial_{j} } \varrho_{\ell } \left( \zeta \right)  \cdot w_{i}\overline{w_{j}} \Big\}
  \\&\geqslant \sum^{k}_{\ell =1} \alpha_{\ell } \inf_{\textbf{z}\in V^{o}_{\Psi,\lambda \left( k+1\right)  }\setminus V^{o}_{\Psi,\lambda \left( k\right)  }} \left( \psi^{\prime } \left( \varrho_{\ell } \left( \textbf{z}\right)  +2-\ell \right)  \cdot \left( -C_{\ell ,k+1}\right)  \right)  \\&=-\sum^{k}_{\ell =1} \alpha_{\ell } C_{\ell ,k+1}\sup_{\textbf{z}\in V^{o}_{\Psi,\lambda \left( k+1\right)  }\setminus V^{o}_{\Psi,\lambda \left( k\right)  }} \left( \psi^{\prime } \left( \varrho_{\ell } \left( \textbf{z}\right)  +2-\ell \right)  \right)   >-\infty,
	\end{aligned}
$$
	where the last inequality follows from the fact that the function $\psi^{\prime } \left( \varrho_{\ell } \left( \textbf{z}\right)  +2-\ell \right)  $ being bounded on $V^{o}_{\Psi,\lambda \left( k+1\right)  }\setminus V^{o}_{\Psi,\lambda \left( k\right)  }$ .
	
	Let
$$\alpha_{k+1} \triangleq \max \left\{ \frac{\left| s_k\right|  }{\varepsilon_{k+1} \psi^{\prime } \left( 1\right)  } ,\frac{\sup\limits_{V_{\Psi,\lambda \left( k+1\right)  }} \varrho}{\psi \left( 1\right)  } \right\}  ,\q\eta _{k+1} \left( \cdot\right)  \triangleq \sum^{k+1}_{j=1} \alpha_{j} \psi \left( \varrho_{j} \left( \cdot\right)  +2-j\right).$$
Then, $\eta _{k+1} \left(  \cdot\right)  \geqslant \eta _{k} \left(  \cdot\right)  \geqslant \varrho \left(  \cdot\right)  $ on $V^{o}_{\Psi,\lambda \left( k\right)  }$, and $$
\begin{aligned}
\eta _{k+1} \left( \textbf{z}\right)  &\geqslant \alpha_{k+1} \psi \left( \varrho_{k+1} \left( \textbf{z}\right)  +2-k-1\right)  \\[3mm]
&\geqslant \alpha_{k+1} \psi \left( \varrho \left( \textbf{z}\right)  +2-k-1\right)  \geqslant \alpha_{k+1} \psi \left( 1\right)  \geqslant \varrho \left( \textbf{z}\right),\q\forall\;\textbf{z}\in V^{o}_{\Psi,\lambda \left( k+1\right)  }\setminus V^{o}_{\Psi,\lambda \left( k\right)  }\subset V\setminus V_{\varrho,k}^{o}.
\end{aligned}
$$
	
	By $\psi^{\left( j\right)  } \geqslant 0$ ($j=0,1,2$), $\varrho_{k+1}$ being plurisubharmonic on $V^{o}_{\Psi,\lambda{(k+1)}}$ and the conclusions $(2)$ and $(3)$ of Proposition \ref{dchcthhshzyxwdxzh23616}, we see that $\alpha_{k+1}\psi \left( \varrho_{k+1} \left( \textbf{z}\right)  +2-k-1\right)  $ is plurisubharmonic on $V^{o}_{\Psi,\lambda{(k+1)}}$, hence so is $ \eta _{k+1} $  on $V^{o}_{\Psi,\lambda{(k)}}$.
	
	For each $n\in \mathbb{N}$, $w_{1},w_{2},\cdots,w_{n}\in \mathbb{C}, \sum\limits^{n}_{i=1} \left| w_{i}\right|^{2}  =1$, we have $$
		\begin{aligned}
			 &\sum_{i,j=1}^n  \partial_{i} \overline{\partial_{j} } \psi \left( \varrho_{k+1} +1-k\right)  w_{i}\overline{w_{j}} \\& =\psi^{\prime \prime } \left( \varrho_{k+1} +1-k\right)  \left| \sum^{n}_{i=1} \partial_{i} \varrho \cdot w_{i}\right|^{2}  +\psi^{\prime } \left( \varrho_{k+1} +1-k\right)  \sum_{i,j=1}^n  \partial_{i} \overline{\partial_{j} } \varrho_{k+1} w_{i}\overline{w_{j}} \\&\geqslant \psi^{\prime } \left( \varrho_{k+1} +1-k\right)  \sum_{i,j=1}^n \partial_{i} \overline{\partial_{j} }  \varrho_{k+1} w_{i}\overline{w_{j}} \\&=\psi^{\prime } \left( \varrho_{k+1} +1-k\right)  \sum_{i,j=1}^n  \partial_{i} \overline{\partial_{j} } \left( \varrho_{k+1} -\varepsilon_{k+1} \left| \left| \textbf{z}\right|  \right|^{2}  \right)  w_{i}\overline{w_{j}}\\
&\q+\psi^{\prime } \left( \varrho_{k+1} +1-k\right)  \sum_{i,j=1}^n  \partial_{i} \overline{\partial_{j} } \left( \varepsilon_{k+1} \left| \left| \textbf{z}\right|  \right|^{2}  \right)  w_{i}\overline{w_{j}} \\&\geqslant \psi^{\prime } \left( \varrho_{k+1} +1-k\right)  \sum_{i,j=1}^n  \partial_{i} \overline{\partial_{j} } \left( \varepsilon_{k+1} \left| \left| \textbf{z}\right|  \right|^{2}  \right)  w_{i}\overline{w_{j}} =\varepsilon_{k+1} \psi^{\prime } \left( \varrho_{k+1} +1-k\right)  \geqslant \varepsilon_{k+1} \psi^{\prime } \left( \varrho +1-k\right)  \\& \geqslant \varepsilon_{k+1} \psi^{\prime } \left( 1\right)
		\end{aligned}
	$$
	on $V^{o}_{\Psi,\lambda {(k+1)}  }\setminus V^{o}_{\Psi,\lambda {(k)}  }\subset V^{o}_{\Psi,\lambda {(k+1)}  }\setminus V^{o}_{\varrho,k}$, where the second inequality comes from the fact that the function $\varrho_{k+1} (\cdot)-\varepsilon_{k+1} \left| \left| \cdot\right|  \right|^{2}  $ is plurisubharmonic on $V^{o}_{\Psi,\lambda{(k+1)}}$.
		Therefore,
$$\sum_{i,j=1}^n  \partial_{i} \overline{\partial_{j} } \eta _{k+1} \left( \textbf{z}\right)  w_{i}\overline{w_{j}} \geqslant s_k+\alpha_{k+1} \varepsilon_{k+1} \psi^{\prime } \left( 1\right) \geqslant s_k+\left| s_k\right|   \geqslant 0,\q\forall\;\textbf{z}\in V^{o}_{\Psi,\lambda {(k+1)}  }\setminus V^{o}_{\Psi,\lambda {(k)}  }.$$ By Lemma \ref{ghqd236130}, we conclude that $\eta _{k+1}$ is plurisubharmonic on $V^{o}_{\Psi,\lambda {(k+1)}  }$.
	
By $\{\varrho_{j}\}_{j=1}^\infty \subset C^{\infty }_{F^{\infty }}\left( V\right)$, it is clear that $\eta _{k+1}   \in C^{\infty }_{F^{\infty }}\left( V\right)  $, hence we complete the construction of the functions $\eta _{k}$ and the positive numbers $\alpha_{k}$ for $k \in \mathbb{N}$.
	
	Let $\eta  \triangleq \lim\limits_{k\rightarrow \infty } \eta _{k} $. Then, for each $k\in \mathbb{N}$ and $\ell \geqslant 3$, by  \eqref{qjhsh236151}, we have
$$\varrho_{k+\ell } \left( \textbf{z}\right)  \leqslant \varrho \left( \textbf{z}\right)  +1,\q\forall\; \textbf{z}\in V^{o}_{\varrho,k}\subset V^{o}_{\Psi,\lambda \left( k\right)  }\subset V^{o}_{\Psi,\lambda \left( k+\ell \right)  },$$ which means that
$$\varrho_{k+\ell } \left( \textbf{z}\right)  +2-k-\ell \leqslant \varrho \left( \textbf{z}\right)  +3-k-\ell <3-\ell \leqslant 0,\q  \forall \;\textbf{z}\in  V_{\varrho,k}^{o}.
$$
Therefore, $\eta _{k+\ell} \left( \textbf{z}\right)  =\eta _{k+2} \left( \textbf{z}\right) $ for all $\textbf{z}\in V^{o}_{\varrho,k}$ and $\ell\geqslant 2$, which leads to $\eta  \left( \textbf{z}\right)  =\eta _{k+2} \left( \textbf{z}\right) $ for each $\textbf{z}\in V^{o}_{\varrho,k}$.
It is then easy to show that $\eta  \in C^{\infty }_{F^{\infty }}\left( V\right)   $ is an exhaustion function on $V$ and satisfies that
$$
		\sum^{n}_{i,j=1} \overline{\partial_{j} } \partial_{i}\eta (\textbf{z}) \zeta_{i}\overline{\zeta_{j}} \geqslant 0,\q\forall\; \textbf{z}\in V,
	$$
for all $ n\in \mathbb{N}$ and $\zeta_{1},\cdots,\zeta_{n}\in \mathbb{C}$. By the conclusion $(d)$ of Lemma \ref{ghqd236130}, we conclude that $\eta  $ is a plurisubharmonic function on $V$.	This completes the proof of Theorem \ref{jntyshxnty23614}.
\end{proof}

Combining Lemma \ref{xntjnt236145} with Theorem \ref{jntyshxnty23614}, we obtain immediately the following consequence, which is the main result of this paper.
\begin{corollary}\label{cor5.1}
$V$ is a pseudo-convex domain of $\ell^{2}$ if and only if there exists an exhaustion function $\eta  \in C^{\infty }_{F^{\infty }}\left( V\right)  $ such that $\eta  $ is plurisubharmonic on $V$ .
\end{corollary}




\begin{thebibliography}{99}


\bibitem{Bre}H. J. Bremermann, \emph{Holomorphic functionals and complex convexity in Banach spaces}. Pacific J. Math., \textbf{7} (1957), 811--831.

\bibitem{Prato2006} G.~Da Prato, \emph{An introduction to infinite-dimensional analysis}. Universitext. Springer-Verlag, Berlin, 2006.

\bibitem{Prato} G. Da Prato and J. Zabczyk, \emph{Second order partial differential equations in Hilbert spaces}. Cambridge University Press Cambridge, 2002.

\bibitem{Dineen} S.~Dineen, \emph{The Cartan-Thullen theorem for Banach spaces}. Ann. Sc. Norm. Super. Pisa Cl. Sci. (5), \textbf{24} (1970), 667--676.

\bibitem{Gro67}L.~Gross, \emph{Potential theory on Hilbert space}. J. Funct. Anal., \textbf{1} (1967), 123--181.

\bibitem{Gruman}L.~Gruman, \emph{The Levi problem in certain infinite dimensional vector spaces}. Illinois J. Math., \textbf{18} (1974), 20--26.

\bibitem{Gruman and Kiselman}L.~Gruman and C.~Kiselman, \emph{Le probl\'eme de Levi dans les espaces de Banach \`a base}. C. R. Acad. Sci. Paris S\'er. A-B, \textbf{274} (1972), A1296--A1299.

\bibitem{Her} M. Herv\'e, \emph{Analyticity in infinite dimensional spaces}. De Gruyter Studies in Mathematics, \textbf{10}. Walter de Gruyter \& Co., Berlin, 1989.

\bibitem{Hirschowitz}A.~Hirschowitz, \emph{Prolongement analytique en dimension infinie}. Ann. Inst. Fourier, \textbf{22} (1972), 255--292.

\bibitem{Hor90}L.~H\"ormander, \emph{An introduction to complex analysis in several variables}. Third edition. North-Holland Mathematical Library, \textbf{7}. North-Holland Publishing Co., Amsterdam, 1990.

\bibitem{LaL86}J. M. Lasry and P.-L. Lions,  \emph{A remark on regularization in Hilbert spaces}. Israel J. Math., \textbf{55} (1986), 257--266.


\bibitem{Mujica}J. Mujica, {\sl Complex analysis in Banach spaces. Holomorphic functions and domains of holomorphy in finite and infinite dimensions}. North-Holland Mathematics Studies, {\bf 120}. Notas de Matem\'atica, {\bf 107}. North-Holland Publishing Co., Amsterdam, 1986.


\bibitem{Noverraz} P. Noverraz, \emph{Pseudo-convexit\'e, convexit\'e polynomiale et domaines d'holomorphie en dimension infinie}. North-Holland Mathematics Studies, {\bf 3}. Notas de Matem\'atica, {\bf 48}. North-Holland Publishing Company, Amsterdam, 1973.

\bibitem{Ohsawa02} T.~Ohsawa,
 \emph{Analysis of several complex variables}.
Translated from the Japanese by S.G.~Nakamura.
Transl. Math. Monogr., {\bf 211}.
Iwanami Ser. Mod. Math.
American Mathematical Society, Providence, RI, 2002.

\bibitem{Pat}I.~Patyi, \emph{On holomorphic domination, II}. C. R. Acad. Sci. Paris, Ser. I, \textbf{353} (2015),  501--503.

\bibitem{Rud87}W.~Rudin,  \emph{Real and complex analysis}. Third edition. McGraw-Hill Book Co., New York, 1987.


\bibitem{WYZ}Z.~Wang, J.~Yu and X.~Zhang, \emph{$L^2$ estimates and existence theorems for the $\overline{\partial}$ operators in infinite dimensions, II}. arXiv:2305.04292.

\bibitem{YZ} J.~Yu and X.~Zhang,   \emph{$L^2$ estimates and existence theorems for the $\overline{\partial}$ operators in infinite dimensions, I}. J. Math. Pures Appl., \textbf{163} (2022), 518--548.

\bibitem{AZ} A.~Zerhusen, \emph{An embedding theorem for pseudoconvex domains in Banach spaces}. Math. Ann., \textbf{336} (2006), 269--280.

\end{thebibliography}
\end{document}